\documentclass[a4paper, reqno]{amsart} 
\usepackage[english]{babel}
\usepackage{a4wide}
\usepackage{graphicx}
\usepackage{amsmath,amssymb,amsthm, amsgen,mathrsfs}
\usepackage{listings}
\usepackage{scalerel}
\usepackage{booktabs}
\usepackage{color}
\usepackage[usenames,dvipsnames]{xcolor}
\usepackage{rotating}
\usepackage{hhline}
\usepackage{multirow}
\usepackage{enumitem}   % replaced enumerate by enumitem for additional feature [resume]
\usepackage[bookmarks]{}
\usepackage{mathtools}
\usepackage{amsfonts}   % if you want the fonts
\usepackage{dsfont}
\usepackage{tikz}
\usepackage{esint}
\usepackage{booktabs}
\usepackage{stmaryrd}
\usepackage{url}
\usepackage{etoolbox} % to add \qed after remark, see below

\numberwithin{equation}{section}
\newtheorem{thm}{Theorem}[section]
\newtheorem{prop}[thm]{Proposition}

\newtheorem{lem}[thm]{Lemma}

\newtheorem{cor}[thm]{Corollary}
\newtheorem{defn}[thm]{Definition}
\theoremstyle{definition}
\newtheorem{rem}[thm]{Remark}
  \AtEndEnvironment{rem}{\null\hfill\qedsymbol}   % uses \usepackage{etoolbox}

\newcommand{\bmu}{\boldsymbol \mu}
\newcommand{\bnu}{\boldsymbol \nu}

\newcommand{\brho}{\boldsymbol \rho}
\newcommand{\bsigma}{\boldsymbol \sigma}

\renewcommand{\div}{\operatorname{div}}
\newcommand{\N}{\mathbb N}
\newcommand{\R}{\mathbb R}
\newcommand{\supp}{\operatorname{supp}}
\newcommand{\weakto}{\rightharpoonup}
\newcommand{\xto}[1]{\xrightarrow{ #1 }}
\newcommand{\xweakto}[1]{ \stackrel{ #1 }{\rightharpoonup} }
\newcommand{\Z}{\mathbb Z}

\newcommand{\e}{\varepsilon}
\newcommand{\bW}{\mathbf W}

\def\T{\mathbb T}
\def\TT{{\scaleto{\T\mathstrut}{6pt}}}
\def\Tt{{\scaleto{\T^2\mathstrut}{7pt}}}
\def\Td{{\scaleto{\T^d\mathstrut}{7pt}}}

\renewcommand\P{\mathcal P}

\DeclareMathOperator\diam{diam}
\DeclareMathOperator\Ent{Ent}

\DeclareMathOperator\Exp{Exp}
\DeclareMathOperator\EExp{EXP}

\def\weakto{\rightharpoonup}

\DeclareFontFamily{U}{mathx}{\hyphenchar\font45}
\DeclareFontShape{U}{mathx}{m}{n}{
      <5> <6> <7> <8> <9> <10>
      <10.95> <12> <14.4> <17.28> <20.74> <24.88>
      mathx10
      }{}
\DeclareSymbolFont{mathx}{U}{mathx}{m}{n}
\DeclareFontSubstitution{U}{mathx}{m}{n}
\DeclareMathAccent{\widecheck}{0}{mathx}{"71}
\begin{document}

\title[Many-particle evolutions with multiple signs]{Convergence and non-convergence of many-particle evolutions with multiple signs}

\author[A. Garroni]{Adriana Garroni} \author[P. van Meurs]{Patrick van Meurs} \author[M.A. Peletier]{Mark A. Peletier} \author[L. Scardia]{Lucia Scardia}

\address[A. Garroni]{Dipartimento di Matematica, Sapienza Universit\`a di Roma, Italy}
\email{garroni@mat.uniroma1.it}

\address[P. van Meurs]{Institute of Liberal Arts and Science, Kanazawa University, Japan}
\email{pjpvmeurs@staff.kanazawa-u.ac.jp}

\address[M.A. Peletier]{Department of Mathematics and Computer Science and Institute for Complex Molecular Systems, Eindhoven University of Technology, The Netherlands}
\email{M.A.Peletier@tue.nl}

\address[L. Scardia]{Department of Mathematics, Heriot-Watt University, United Kingdom} 
\email{L.Scardia@hw.ac.uk}

\begin{abstract} 
We address the question of convergence of evolving interacting particle systems as the number of particles tends to infinity. We consider two types of particles, called positive and negative. Same-sign particles repel each other, and opposite-sign particles attract each other. The interaction potential is the same for all particles, up to the sign, and has a logarithmic singularity at zero. The central example of such systems is that of dislocations in crystals.

Because of the singularity in the interaction potential, the discrete evolution leads to blow-up in finite time. We remedy this situation by regularising  the interaction potential at a length-scale $\delta_n>0$, which converges to zero as the number of particles $n$ tends to infinity.

We establish two main results. The first one is an evolutionary convergence result showing that the empirical measures of the positive and of the negative particles converge to a solution of a set of coupled PDEs which describe the evolution of their continuum densities. In the setting of dislocations these PDEs are known as the Groma-Balogh equations. In the proof we rely on both the theory of $\lambda$-convex gradient flows, to establish a quantitative bound on the distance between the empirical measures and the continuum solution to a $\delta_n$-regularised version of the Groma-Balogh equations, and \emph{a priori} estimates for the Groma-Balogh equations to pass to the small-regularisation limit in a functional setting based on Orlicz spaces. In order for the quantitative bound not to degenerate too fast in the limit $n \to \infty$ we require $\delta_n$ to converge to zero sufficiently slowly. 
The second result is a counterexample, demonstrating that if $\delta_n$ converges to zero sufficiently \emph{fast}, then the limits of the empirical measures of the positive and the negative dislocations do not satisfy the Groma-Balogh equations. 

These results show how the validity of the Groma-Balogh equations as the limit of many-particle systems depends in a subtle way on the scale at which the singularity of the potential is regularised.
\end{abstract}

\maketitle

\section{Introduction}
\label{sec:Introduction}

\subsection{Multi-sign particle systems and dislocations}

There is a vast literature on the properties of interacting particle systems, both deterministic and stochastic, as well as on their limits as the number of particles tends to infinity. The most common situation in the rigorous mathematical literature is that of indistinguishable particles. For many systems of this type, especially with bounded interactions,  the many-particle limit has been characterised in various ways; see e.g.~\cite{Ruelle69,Spohn80,Spohn91,KipnisLandim99,BraidesGelli-lr-02} for reviews. For unbounded interaction forces the situation is more delicate, and rigorous proofs of convergence are more recent~\cite{Ponsiglione07,SS-book07, Hauray09,CarrilloFerreiraPrecioso12,CarrilloChoiHauray14,GeersPeerlingsPeletierScardia13,ScardiaPeerlingsPeletierGeers14,VanMeursMuntean14,VanMeursMunteanPeletier14,AlicandroDe-LucaGarroniPonsiglione14,MoraPeletierScardia14TR}.

In this paper we study a type of particle system that is less intensively studied: a system with two species, called positive and negative particles \cite{GarroniLeoniPonsiglione10,
ChapmanXiangZhu15, 
DiFrancescoFagioli13, 
BerendsenBurgerPietschmann17, 
EversFetecauKolokolnikov17,
vanMeurs17}. In our setting, we also consider singular interactions. Such systems arise naturally as models of the evolution of \emph{dislocations} in crystals. Since this example is important for our work, we now explain it in some detail.

\medskip

Dislocations are defects in an atomic lattice, and they are central to the theory of plastic deformation. A dislocation can be viewed as a \emph{quantum of plastic slip:} the smallest amount of plastic deformation that the lattice admits. Macroscopic, continuum-scale plastic deformation is the collective result of the motion of a large number of dislocations. 

Models of plastic deformation vary depending on the scale at which they describe the system. At scales of millimeters or larger, plastic deformation is well described by continuum-level theories (see e.g.~\cite[Ch.6]{Callister07}). 
However, these continuum-level theories often break down at scales of $1{-}100$ $\mu$m
where the length scales of the specimen size and the dislocation distribution become comparable. At such smaller scales, a richer set of unknowns is used,  the \emph{densities of dislocations}, and  equations are formulated for the evolution of these densities. At even smaller scales, the postulate of a `smooth density' fails, and a description becomes necessary in terms of the individual dislocations and their motion: such systems are called \emph{discrete dislocation} systems. This paper is concerned with the transition between the latter two types of models: from discrete dislocations to dislocation densities, as the number of dislocations becomes large. 

A popular evolution model for  dislocation densities was derived by Groma and Balogh in \cite{Groma97,GromaBalogh99} and later refined in \cite{GromaCsikorZaiser03,GromaGyorgyiKocsis06}. It describes the evolution of the density of `positive' and `negative' straight and parallel edge dislocations, which are represented by `positive' and `negative' points in two dimensions. This model, and elaborations of it~\cite{GromaVandrusIspanovity15,GromaZaiserIspanovity16} have been used in the engineering community to predict dislocation density profiles with surprising accuracy~\cite{YefimovGromaVanderGiessen04,YefimovVanderGiessen05a,GeersPeerlingsHoefnagelsKasyanyuk09,DoggePeerlingsGeers15a}. The original system of equations from~\cite{GromaBalogh99} is central to this paper; we will refer to it as the \emph{Groma-Balogh} equations, and it appears in generalized form as equations~\eqref{f:GBintro} below.

The derivations in~\cite{Groma97,GromaBalogh99,GromaCsikorZaiser03,GromaGyorgyiKocsis06} are based on an upscaling argument starting from a system of discrete dislocations. None of these results are rigorous, however, since they build on uncontrolled approximations such as exchanging averaging with nonlinearity~\cite{GromaBalogh99} or postulating a closure relation in the BBGKY hierarchy~\cite{Groma97,GromaZaiserIspanovity16}.

\medskip
This brings us to the central question of this paper:
\begin{quote}
Can the Groma-Balogh equations be \emph{rigorously} derived from an underlying discrete-dislocation system?
\end{quote}

We will see that the answer is subtle, and depends on how one decides to deal with the singularity of the interaction potential. 

\subsection{Formal setup}

The mathematical results of this paper are formulated for two different settings: a $d$-dimensional torus with isotropic evolution (where $d=1,2$ are the ones inspired by screw dislocations), and a two-dimensional torus with slip-plane confinement modelling edge dislocations. We describe these two situations and their physical interpretation in Section~\ref{sec:setups}. At this stage, however, we  illustrate the results just for the two-dimensional case with isotropic evolution.
We work in a two-dimensional square torus $\T^2$, or equivalently a square in $\R^2$ with periodic boundary conditions. 

We model dislocations as points in $\T^2$; they represent defects in a linearly elastic continuum medium. 
Systems of dislocations have an associated energy, which is the elastic energy of the stress and strain fields in the elastic medium, generated by the defects and by external loads. Formally, the quadratic nature of the continuum elastic energy leads to an expression of the form~\cite{CermelliLeoni05,MoraPeletierScardia14TR}
\begin{align} \label{eqn:disc:energy:unregzd}
  \tilde E_n (x; b) 
  := \frac1{n^2} \sum_{i = 1}^n \sum_{\substack{ j = 1 \\ j \neq i } }^n b_i b_j V (x_i - x_j)
   + \frac1n \sum_{i=1}^n b_i U(x_i).
\end{align}
Here $x=(x_1,\dots,x_n)\in (\T^2)^n$ and $b=(b_1,\dots,b_n)\in \{-1,+1\}^n$ denote the position and the sign of each dislocation, $U$ is a smooth external potential, and $V(x-y)$ is a pairwise interaction potential which characterises the dependence of the elastic energy on the distance $x-y$ between two dislocations. The potential $V$, which will be defined in Section~\ref{sec:setups}, is related to the Green's function of the elasticity operator. Consequently, it has a logarithmic singularity at the origin with $V(0)=+\infty$. The logarithmic singularity is central to this paper, and we will deal with it in detail in Section~\ref{s:intro:Vreg} below. 
\smallskip

We now describe the main question of the paper. Assuming that the velocities of the dislocations satisfy an isotropic linear drag law (Orowan's relation), their evolution in $(\T^2)^n$ is of gradient-flow type, driven by the energy~\eqref{eqn:disc:energy:unregzd}, namely
\begin{equation} \label{for:evo:dyncs:edge:2D:unregzd}
     \frac{d x}{d t} (t) = - n \nabla \tilde E_n (x(t); b).
\end{equation}
By translating equation~\eqref{for:evo:dyncs:edge:2D:unregzd} into the language of measures, we find that the empirical measures of the positive and the negative dislocations,
\begin{equation}
\label{def:mu-plus-minus}
\mu_n^+ := \frac1n \sum_{\substack{i=1\\b_i = +1}}^n \delta_{x_i}, 
\qquad 
\mu_n^- := \frac1n \sum_{\substack{i=1\\b_i = -1}}^n \delta_{x_i}, 
\end{equation}
formally satisfy  the aforementioned \emph{Groma-Balogh equations}\footnote{%
This well-known argument runs as follows: for a test function $\varphi$, we have
\begin{align*}
\frac d{dt} \int_{\T^2} \varphi(x)\mu_n^+(dx) &= \frac d{dt} \frac1n \sum_{\substack{i=1\\b_i = +1}}^n \varphi(x_i) 
= -\frac1n \sum_{\substack{i=1\\b_i = +1}}^n \nabla\varphi(x_i) \biggl[\frac1n \sum_{j=1}^n b_ib_j\nabla V(x_i-x_j) + b_i \nabla U(x_i)\biggr]\\
&= -\int_{\T^2} \mu_n^+(dx) \nabla \varphi(x)\biggl[ \int_{\T^2} \nabla V(x-y) (\mu_n^+-\mu_n^-)(dy) + \nabla U(x)\biggr],
\end{align*}
which is a weak formulation of equation~\eqref{f:GBintro-1}.
}
\begin{subequations}
\label{f:GBintro}
\begin{align} 
\label{f:GBintro-1}
\partial_t \rho^+ 
&= \phantom{-}\div \big( \rho^+ ( \nabla V * (\rho^+ - \rho^-) + \nabla U ) \big),\\
\partial_t \rho^- 
&= - \div \big( \rho^- ( \nabla V * (\rho^+ - \rho^-) + \nabla U) \big).
\end{align}
\end{subequations}
This would suggest that as $n\to\infty$, if $\mu_n^\pm$ converge to limits $\rho^\pm$, then these limits should also be solutions of the same set of equations~\eqref{f:GBintro}.
The main aim of this paper is to investigate this convergence.

However, the formal arguments above cannot be made rigorous, because the singularity in $V$ results in existence of solutions to \eqref{for:evo:dyncs:edge:2D:unregzd} only up to the first time two particles of opposite sign collide. To resolve this blow-up, we replace \eqref{for:evo:dyncs:edge:2D:unregzd} by a regularised version, which we introduce next.

\subsection{Managing the singularity}
\label{s:intro:Vreg}

The singularity of $V$ has its physical origin in the simplification of replacing the crystallographic lattice by a continuum linearly elastic medium and modelling the dislocations as point defects. These point defects generate stress and strain fields with infinite elastic energy in any neighbourhood of the defect. The energy of a single dislocation or of a finite set of dislocations -- which in a genuine discrete model should be finite -- is therefore infinite\footnote{This is analogous to the fact that the fundamental solution $G$ of the Laplacian in two dimensions has a singularity at zero and has infinite Dirichlet integral $\int_O |\nabla G|^2 $ in any neighbourhood $O$ of zero, while by contrast the fundamental solution of the discrete Laplacian on a lattice is bounded. }.

This infinite energy makes many energy-based methods unsuitable, and several approaches and alternative models have been developed to circumvent this difficulty:
\begin{enumerate}
  \item the phase-field model developed by Peierls and Nabarro \cite{Peierls40, Nabarro47, KoslowskiCuitinoOrtiz02, GarroniMueller06, MonneauPatrizi12},
  \item the removal of small balls around the dislocations from the elastic medium \cite{CermelliLeoni05, GarroniLeoniPonsiglione10, MoraPeletierScardia14TR},
  \item \label{item:regularization-by-convolution} the smearing out of the dislocation core by a convolution kernel \cite{AlvarezCarliniHochLBouarMonneau05, CaiArsenlisWeinbergerBulatov06, GarroniLeoniPonsiglione10, ContiGarroniOrtiz},
  \item a cut-off radius within which dislocations  do not interact \cite{HirthLothe82}. 
\end{enumerate} 
The choice of the regularisation depends on factors such as: accuracy, computational convenience, the possibility to describe dynamics, well-posedness or the possibility to prove discrete-to-continuum convergence.

\medskip

In this paper we consider a broad class of regularisations that includes case~\eqref{item:regularization-by-convolution} above: we replace the potential $V$ by a smoothed, globally $W^{2,\infty}$ potential $V_\delta$, where $\delta>0$ is a parameter with the interpretation of the length scale of the regularisation (see assumptions~(V1)-(V4) in Section~\ref{S:setup}). In view of \eqref{f:GBintro}, we are interested in the limit $\delta \to 0$.

The regularised energy (denoted by $E_n$ without the tilde) is defined as~\eqref{eqn:disc:energy:unregzd}, with~$V$ replaced by $ V_{\delta}$ and with the diagonal kept in the sum:
\begin{align} \label{eqn:disc:energy}
  E_n (x; b) 
  := \frac1{2 n^2} \sum_{i = 1}^n \sum_{j = 1}^n b_i b_j V_{\delta} (x_i - x_j)
     + \frac1n \sum_{i=1}^n b_i U(x_i).
\end{align}
Similarly to~\eqref{for:evo:dyncs:edge:2D:unregzd} the dynamics are now given by 
\begin{equation} \label{for:evo:dyncs:edge:2D}
     \frac{d x}{d t} (t) = - n \nabla  E_n (x(t); b).
\end{equation}
Note that since $\nabla V_{\delta}$ is globally Lipschitz continuous, this evolution is well-defined. 
\smallskip

By translating~\eqref{for:evo:dyncs:edge:2D} into measures, we now find \emph{rigorously} that the empirical measures $\mu_n^\pm$ in \eqref{def:mu-plus-minus}, associated with a solution $x$ of~\eqref{for:evo:dyncs:edge:2D}, satisfy the \emph{regularised} Groma-Balogh equations
\begin{subequations}
\label{f:GBdel-intro}
\begin{align} 
\partial_t \rho^+ 
&= \phantom{-}\div \big( \rho^+ ( \nabla V_{\delta} * (\rho^+ - \rho^-) + \nabla U ) \big),\\
\partial_t \rho^- 
&= - \div \big( \rho^- ( \nabla V_{\delta} * (\rho^+ - \rho^-) + \nabla U) \big),
\end{align}
\end{subequations}
in the weak sense. These equations are well-defined in the sense of distributions on $(0,T)\times \T^2$ whenever $\rho^\pm$ are finite-mass measures (such as $\mu_n^\pm$). Indeed, since $\nabla V_{\delta}$ is uniformly continuous on $\T^2$, so is the convolution $\nabla V_{\delta}*(\rho^+-\rho^-)$, and therefore $\rho^\pm \nabla V_{\delta}*(\rho^+-\rho^-)$ is a finite-mass vector-valued measure on $\T^2$ for each $t$. This argument cannot be applied to the unregularised Groma-Balogh equations \eqref{f:GBintro}. There are, however, several ways to define a solution concept to these equations. We continue this discussion in \S \ref{s:intro:disc}.

\subsection{Results of this paper: Convergence and non-convergence}
We first describe our \emph{convergence} result. When the regularisation $\delta$ is \emph{fixed}, i.e., when we are dealing with a fixed, regular interaction kernel, the convergence of the discrete evolution equation~\eqref{for:evo:dyncs:edge:2D} to the Groma-Balogh equations~\eqref{f:GBdel-intro} as $n \to \infty$ is standard; the technique goes back at least to Dobrushin~\cite{Dobrushin79}. Here, however, we will consider the joint limit $n \to \infty$, $\delta_n \to 0$, to which the standard theory does not apply.

Our main convergence result states that when $\delta_n\to0$ sufficiently slowly, convergence to the unregularised Groma-Balogh equations \eqref{f:GBintro} holds. Theorem~\ref{t:evoConv} specifies this result for the $d$-dimensional torus, and Theorem~\ref{t:evoConv:conf} specifies it for the two-dimensional torus with slip-plane confinement.

The meaning of `how slowly' $\delta_n$ needs to converge to zero for the result to hold depends both on the chosen time horizon $T>0$, and on the sequence of discrete initial data for~\eqref{for:evo:dyncs:edge:2D} approximating the initial datum of \eqref{f:GBintro}. Conversely, we also specify a lower bound on $\delta_n$ such that our convergence result holds for a certain class of approximating initial data. In particular, if $\delta_n \gg n^{-1/d}$, where $n^{-1/d}$ is the typical distance between neighbouring dislocations, we recover the Groma-Balogh equations in the limit. Note that in this regime the short-range interactions are governed by the model of the dislocation core, whose size is described by $\delta_n$; the limiting Groma-Balogh equations, however, are independent of the description of the core. 

Our second main result is a counterexample, showing \emph{non-convergence} of the discrete evolutions to \eqref{f:GBintro} when $\delta_n\to 0$ \emph{fast enough}, for a suitable choice of initial conditions. The key idea is to choose as discrete initial data a configuration of short dipoles whose mutual distance is infinitesimal as $n\to \infty$, but much larger than the inter-dipole distance. In this case the discrete  evolutions converge to limit measures $\rho^\pm$ that are constant in space and time. This construction provides a counterexample to the discrete-to-continuum convergence whenever~$U$ is non-constant, since in that case, on the torus, spatially-constant measures cannot be stationary solutions of the Groma-Balogh equations~\eqref{f:GBintro}. 

\subsection{Comments and perspectives}
\label{s:intro:disc}

We comment on a number of aspects of this work and on some natural questions triggered by our results.

\smallskip

\emph{Conditions on the potential.} 
The conditions (V1)-(V4) that we impose below on the interaction potential $V$ cover many examples in materials science, such as vortices, Coulomb gases and dislocations. They are a mixture of fundamental conditions and conditions that we believe are mostly technical. The fundamental assumptions are the positivity of the Fourier transform, the singularity of $V$ being at most logarithmic, and the convergence of the approximations. The logarithmic singularity of $V$ is the strongest singularity under which the proof of the well-posedness of \eqref{f:GBintro} in \cite{CannoneEl-HajjMonneauRibaud10} (that we rely on in our proof) holds, and we are unaware of results that can deal with a singularity stronger than the logarithmic one. Positivity of the Fourier transform of the regularisation $V_\delta$ of the potential $V$ is needed in Lemma \ref{l:V:props}, where we establish a bound for the convolution with $V_\delta$ that is crucial in the proof of the evolutionary convergence. Moreover, for regularised potentials $V_\delta$ that do not have positive Fourier transform, the numerical simulations in van Meurs' thesis~\cite[Ch.~9]{VanMeurs15TH} show that \eqref{for:evo:dyncs:edge:2D} can have a fundamentally different behaviour.  
Note that positivity of the Fourier transform also guarantees convexity of $\kappa \mapsto \int V*\kappa \, d\kappa$, which for $\kappa = \rho^+-\rho^-$ is the interaction energy associated to \eqref{f:GBintro}. 

\smallskip

\emph{Regularisation and a-priori estimates.} There is a useful rule of thumb for the analysis of the properties of nonlinear PDEs, going back to the work of Jacques-Louis Lions and the French School of nonlinear PDE: `if the right a-priori estimates can be established, then one can regularise any way one wants---without changing the results'. In other words, the \emph{equation} determines the solutions, through the a-priori estimates, and regularisations can only approximate such solutions. Even in cases where uniqueness is not obvious, such as the incompressible Navier-Stokes equations, this rule of thumb applies.

This rule-of-thumb requires, however, that \emph{the approximations do not violate the a-priori estimates}. It is exactly this requirement that is not satisfied in this paper, and this is easy to recognize. The existence theorems, for slow $\delta_n\to0$ (Theorems~\ref{t:evoConv} and \ref{t:evoConv:conf}) build upon the a-priori estimates first identified in~\cite{CannoneEl-HajjMonneauRibaud10}. These estimates are meaningful for solutions with bounded entropy, i.e.,\ solutions $\rho^\pm$ for which $\int\rho^+\log\rho^+$ and $\int \rho^-\log \rho^-$ are bounded. Empirical measures do not satisfy this condition, and therefore a `regularisation' based on point measures, while physically meaningful, lies outside of this functional setting. 

We nonetheless make use of the entropy-based estimates by using the fact that the regularised evolution has a bounded expansion rate in the Wasserstein metric, with a bound that deteriorates as $\delta_n\to0$. By constructing an intermediate object (i.e., the unique solution of the $\delta_n$-regularised Groma-Balogh equations \eqref{f:GBdel-intro} with a continuum initial datum, which is both entropically bounded and Wasserstein-close to the empirical measures), we can extend the entropic framework a little outside the finite-entropy realm. The price to pay is that the Wasserstein bound need not deteriorate too fast, and this translates into the slowness criterion of $\delta_n\to0$.

\smallskip

\emph{Beyond regularising the dislocation core}. 
As an alternative to regularising the singular potential $V$, one can consider \emph{annihilation} of dislocations, where dislocations of opposite signs are taken out of the system when they collide or are sufficiently close. This evolution is studied in \cite{Serfaty07II,SmetsBethuelOrlandi07} for finite $n$, in \cite{Head72III,bilerKarchMonneau10,Ambrosio2011} on the continuum level, and in \cite{vanMeursMorandotti18ArXiv} a discrete-to-continuum convergence result is established. The main difference between the models listed above and the Groma-Balogh equations \eqref{f:GBintro} is in the reversibility/irreversibility of the annihilation. In \eqref{f:GBintro}, if a positive and a negative dislocation meet, they form a short dipole, but if a sufficiently large force is applied to the system, they can break apart. In other words, annihilation is `reversible'. In the other models, instead, annihilation is irreversible: once a short dipole forms, the two dislocations in it become permanently `invisible' in the evolution.

\smallskip

\emph{Mixed-approach proof of convergence.} Our proof of convergence is made in two steps, each requiring a different strategy. To estimate the distance between the empirical measures and the continuum solution of the $\delta_n$-regularised Groma-Balogh equations, we exploit the gradient-flow structure of the regularised equations of which they are both solutions (although with different initial data), and prove a Gronwall-type estimate in terms of Wasserstein distances. To estimate the distance between the intermediate measure and a solution of the unregularised Groma-Balogh equations we instead follow the approach in \cite{CannoneEl-HajjMonneauRibaud10}, based on entropy estimates.

\smallskip

\emph{Well-posedness for the Groma-Balogh equations \eqref{f:GBintro}.} Various existence results for the Groma-Balogh equations have already been proved~\cite{ElHajj07,ElHajj10,CannoneEl-HajjMonneauRibaud10,Mainini12a,LiMiaoXue14,WanChen16,WanChen17TR}, and other related multiple-sign systems have been studied in~\cite{ElHajjForcadel08,AmbrosioSerfaty08,Ambrosio2011}. A by-product of the convergence theorems in this paper provides additional existence results, which go beyond those cited above in allowing for more general interaction potentials.

The questions of uniqueness and stability appear to depend strongly on the regularity of the initial data. Mainini~\cite{Mainini12a} proves uniqueness for a similar system in two dimensions with different boundary conditions, under the assumption that the solutions are in $L^\infty$, by using the log-Lipschitz continuity of $\nabla V*(\rho^+-\rho^-)$.  Li, Miao, and Xue \cite{LiMiaoXue14} prove local-in-time uniqueness in $H^m\cap L^p$, $m>2$ and $p\in (1,2)$, and show that a finite existence time in these spaces implies blow-up in $L^\infty$.

\smallskip

\emph{Non-convergence}. Our results fit into the setting of evolutionary convergence of multiple species interacting via singular potentials, e.g., \cite{ChapmanXiangZhu15}, \cite[Ch.~9]{VanMeurs15TH}, \cite{vanMeurs17}. On the other hand, our counterexample shows that the discrete system does not always converge to the expected limiting equation, or worse, may not converge at all. This sparks questions for future research such as:
\begin{itemize}
 \item[-] is there an alternative notion of evolutionary convergence (e.g., statistical mechanics, convergence in probability on random initial data, addition of noise), weaker than the one in Theorem \ref{t:evoConv}, for which convergence can be proven?
 \item[-] are there microscopic details invisible on the macroscale (such as the density of dipoles) which can affect the macroscopic behaviour, rendering the question of the discrete-to-continuum convergence much more subtle?
\end{itemize}

\medskip

The organisation of the paper is as follows. In Section \ref{sec:setups} we give a precise description of \eqref{for:evo:dyncs:edge:2D}, \eqref{f:GBdel-intro} and \eqref{f:GBintro} for both the isotropic and slip-plane confined evolutions, a detailed description of the assumptions and related properties of the interaction potential $V$ and its regularisation $V_\delta$, and further preliminaries. In Section \ref{sec:convergence} we state and prove our main convergence results (Theorem \ref{t:evoConv} and Theorem \ref{t:evoConv:conf}). Our second main result, the class of counterexamples for convergence to the Groma-Balogh equations, is detailed in Section \ref{section:dipole}. The Appendix \ref{sect:Orlicz} recalls the definitions and elementary properties of several Orlicz spaces.

%%%%%%%%%%%%%%%%%%%%%%%%%%%%%%%%%%%%%%%%%%%%%%%%%%%%%

\section{Detailed formulation of the problems and preliminary results}
\label{sec:setups}

\noindent
Before describing the isotropic and anisotropic cases in detail, we first introduce the general setup in dimension $d\geq 1$, and prove some preliminary results that apply to both cases.

\subsection{Notation}

\noindent

\begin{minipage}{16cm}
\newcommand{\specialcell}[2][c]{%
  \begin{tabular}[#1]{@{}l@{}}#2\end{tabular}}
\begin{small}
\bigskip
\begin{tabular}{lll}
$\mathcal M(\Omega)$, $\mathcal M_+(\Omega)$ & signed and non-negative finite Borel measures on a generic set~$\Omega$;\\
$m\P(\Omega)$, $m>0$ & non-negative Borel measures of mass $m$ on a generic set~$\Omega$\\
&  (see also p.~\pageref{remark:notationP} for the notation $\P (\T^d \times \{\pm1\} )$);\\
$\T^d$, $d\in \N$ & $d$-dimensional open flat torus, $\T^d = \R^d/\Z^d$;\\
$\widehat f_k$, $k\in \Z^d$ & \specialcell[t]{Fourier coefficients of $f\in L^2(\T^d)$ ($\widehat f =(\widehat f_k)_k \in \ell^2(\Z^d)$), \\ 
$\widehat f_k := \int_{\mathbb T^d} e^{-2 \pi i k \cdot x} f(x) \, dx$;}\\
$\mathcal{F}^{-1}(g)$ & \specialcell[t]{inverse Fourier transform of $g=(g_k)_k \in \ell^2(\Z^d)$,\\
\smallskip
$\mathcal{F}^{-1}(g)(x) := \sum_{k \in \Z^d} g_k e^{2 \pi i k \cdot x}$;}\\
\smallskip
$ \| g\|_{H^\ell (\mathbb T^d)}$, $ \ell \in \R$ & \specialcell[t]{norm of $g\in H^\ell (\mathbb T^d)$,\\
$ \|g\|^2_{H^\ell (\mathbb T^d)}:= \sum_{k \in \Z^d} (1 + |k|^2)^\ell \big| \widehat g_k \big|^2$;}\\
$W(\mu,\nu)$ & \specialcell[t]{$2$-Wasserstein distance between $\mu,\nu \in\mathcal M_+ (\Omega)$ \\
(set $=+\infty$ if the measures have different masses);}\\
$\| f \|_{\textrm{BL}}$ & \specialcell[t]{norm on bounded Lipschitz functions $f:\Omega\to \R$,\\
$\| f \|_{\textrm{BL}} = \|f\|_\infty + |f |_{\textrm{Lip}}$, and $|f |_{\textrm{Lip}}$ the Lipschitz constant of $f$;}\\
$\| \mu \|_{\textrm{BL}}^*$ & \specialcell[t]{dual bounded Lipschitz norm of $\mu \in \mathcal M(\Omega)$,\\
$\| \mu \|_{\textrm{BL}}^* = \sup\big\{|\int_\Omega f d\mu|: \|f\|_{\textrm{BL}}\leq 1\big\}$;}\\
$\Ent (\mu)$  & entropy of $\mu \in \mathcal M_+ (\Omega)$, $\Ent (\mu):=\int_{\Omega} \mu \log \mu$;\\
$\Exp_\alpha (\mathbb T^d)$, $L\log^\beta L(\T^d)$ & Orlicz spaces on the flat torus (see Appendix \ref{sect:Orlicz}).\\
\end{tabular}
\end{small}
\end{minipage}

\subsection{General setup}\label{S:setup}
We consider points $x_i\in \T^d$ with fixed signs $b_i\in \{-1,+1\}$ for $i=1,\dots,n$ that evolve according to \eqref{for:evo:dyncs:edge:2D}.
We assume the following conditions on the related potentials $V$, $V_\delta$, and $U$: 
\smallskip

\begin{itemize}
\item[(V1)] $\widehat V_k \geq 0, \,\,  [{\widehat V_\delta}]_k \geq 0 \,\, \forall \,  k\in \Z^d\setminus \{0\}$;
\smallskip
\item[(V2)] $\sup_{k\in \Z^d\setminus \{0\}}(1+|k|^2)^\frac{d}2 \big(\widehat V_k\vee [{\widehat V_\delta}]_k\big) < \infty$ uniformly in $\delta$;
\smallskip
\item[(V3)] $V_\delta \to V$ in $\mathcal D' (\mathbb T^d)$ and $V_\delta \ast \phi \to V\ast \phi$ strongly in $\Exp(\T^d)$ for every $\phi \in C^1(\T^d)$, as $\delta \to 0$;
\smallskip
\item[(V4)] $V_\delta \in W^{2,\infty}(\T^d)$ with $V_\delta(x) = V_\delta(-x)$;
\smallskip
\item[(U)] $U\in C_b^\infty(\T^d)$.
\end{itemize}

We refer to Appendix~\ref{sect:Orlicz} for the definition of the Orlicz space $\Exp (\T^d)$.

\begin{rem}
In the two-dimensional case, a possible choice for $V$ is the Green's function on the flat torus $\T^2$. Its Fourier coefficients are given by
\begin{align*}
\widehat V_k = \begin{cases}
\frac\alpha{|k|^2} & k\in \Z^2\setminus\{0\}\\
0 & k=0
\end{cases} \qquad \text{for }\alpha>0.
\end{align*}
A possible choice for its regularisation $V_\delta$ is
\begin{align*}
[\widehat V_\delta]_k &= \begin{cases}
\frac\alpha {|k|^2} e^{-\delta |k|} & k\in \Z^2\setminus\{0\}\\
0 & k=0.
\end{cases}
\end{align*}
\end{rem}

\noindent
Given a point $x =(x_1,\dots,x_n)\in (\mathbb T^d)^n$ and the associated sign $b =(b_1,\dots,b_n)\in \{\pm 1\}^n$, the discrete, regularised energy of the system is 
\begin{align} \label{eqn:disc:energy-repeated}
  E_n (x; b) 
  := \frac1{2 n^2} \sum_{i = 1}^n \sum_{j = 1}^n b_i b_j V_{\delta} (x_i - x_j)
     + \frac1n \sum_{i=1}^n b_i U(x_i),
\end{align}
where $\delta> 0$ (see also~\eqref{eqn:disc:energy}). For what follows it is convenient to express the particle energy \eqref{eqn:disc:energy-repeated} in terms of the empirical measures associated to $x$ and $b$. 

We recall the definition of empirical measures (see~\eqref{def:mu-plus-minus})
\begin{equation} \label{fd:mun}
\begin{aligned}
   \mu^+_n &:= \frac1n \sum_{\substack{i=1\\b_i = +1}}^n \delta_{x_i} \in \mathcal M_+ (\T^d),
   &\mu^-_n &:= \frac1n \sum_{\substack{i=1\\b_i = -1}}^n \delta_{x_i} \in \mathcal M_+ (\T^d), \\
   \bmu_n &:= (\mu^+_n, \mu^-_n) \in (\mathcal M_+ (\T^d))^2,
   &\kappa_n &:= \frac1n \sum_{i=1}^n b_i \delta_{x_i} = \mu_n^+ - \mu_n^- \in \mathcal M (\T^d).
\end{aligned}
\end{equation}
Note that there is an obvious isomorphism between $(x,b)$ and $\bmu_n$, modulo relabelling the particles. 
\label{remark:notationP}
Throughout the paper we will consider the pairs of measures $\bmu=(\mu^+,\mu^-) \in (\mathcal M_+ (\T^d))^2$ such that $\mu^++\mu^- \in \P(\T^d)$. With a slight abuse of notation we will denote this class with $\P (\T^d \times \{\pm1\} )$ (implicitly identifying the measure $\bmu$ with a measure $\mu^++\mu^-\in \P (\T^d \times \{\pm1\} )$, with $\supp \mu^+\subseteq \T^d \times \{+1\}$ and  $\supp \mu^-\subseteq \T^d \times \{-1\}$).

\medskip

We define a continuum, $\delta$-regularised energy $E_\delta$, that extends \eqref{eqn:disc:energy-repeated} to the whole space of probability measures $\P (\T^d \times \{\pm1\})$, as
\begin{equation*}
E_\delta(\bmu):= \frac12\int_{\T^d} (V_\delta\ast \kappa )(x) d \kappa (x) + \int_{\T^d} U(x) d \kappa (x),
\end{equation*}
where, in analogy with \eqref{fd:mun}, $\kappa:= \mu^+-\mu^-$. Since the energy $E_n$ is invariant under relabelling the particles, we have $E_\delta(\bmu_n) = E_n(x;b)$, where $\bmu_n$ is the empirical measures associated to $x$ and $b$.

\medskip

We also recall the definition of entropy: for any $\brho=(\rho^+,\rho^-) \in \P(\T^d\times \{\pm1\})$ with $\rho^\pm \ll dx$, 
\begin{equation}\label{def:entropy}
\Ent(\brho) := \Ent(\rho^+) + \Ent (\rho^-) := \int_{\T^d} \rho^+(x)\log \rho^+(x) dx + \int_{\T^d} \rho^-(x)\log \rho^-(x) dx.
\end{equation}
With a little abuse of notation, here, and in what follows, we denote by $\rho^+$ (resp.~$\rho^-$) a positive measure and its density with respect to the Lebesgue measure.

\medskip

Finally we introduce the Wasserstein distance between measures in $\mathcal{P}(\T^d\times\{\pm 1\})$.
Let $\bmu, \bnu \in \mathcal{P}(\T^d\times\{\pm 1\})$. We define the (square of the) 2-Wasserstein distance between $\bmu$ and $\bnu$ as 
\begin{align} \label{bW2}
\bW^2 \big( \bmu, \bnu \big) := \inf_{\gamma \in \Gamma(\bmu,\bnu)} \int_{(\T^d\times\{\pm1\})^2} d^2(x',y') d\gamma(x',y'),
\end{align}
where,  for $x', \ y' \in \T^d\times\R$, with  $x'=(x,a)$ and $y'=(y,b)$, 
\begin{equation}\label{d2}
d^2(x',y'):= \|x-y\|_{\Td}^2 + |a-b|^2.
\end{equation}
Here $\|\cdot \|_{\Td}$ (denoted with $|\cdot|_{\TT}$ in the case $d=1$) is the induced metric on the manifold $\T^d$,
\begin{equation}\label{norm:T}
\|x-y\|_{\Td} = \min_{k\in \Z^d}\left\{ \|x-y +k\| \right\},
\end{equation}
and $\Gamma(\bmu,\bnu)$ is the set of couplings of $\bmu$ and $\bnu$, namely
\begin{align*}
\Gamma(\bmu,\bnu) := &\Big\{\gamma\in \P\big((\T^d\times\{\pm1\})^2\big): \gamma(A\times (\T^d\times\{\pm1\})) = \bmu(A), \\
  & \quad \gamma((\T^d\times\{\pm1\})\times A) = \bnu(A)\text{ for all Borel sets }A\subset \T^d\times\{\pm1\} \Big\}.
\end{align*}
As usual, we denote with $\Gamma_\circ (\bmu,\bnu) \subseteq \Gamma(\bmu,\bnu)$ the set of \emph{optimal transport plans} $\gamma$ for \eqref{bW2}. Note that $\Gamma_\circ (\bmu,\bnu) \neq \emptyset$ (see, e.g., \cite[Theorem 1.4]{Santambrogio-book}).

\smallskip

In the special case of $\mu^\pm (\T^d) = \nu^\pm (\T^d)$, $\bW$ enjoys some additional properties summarised in the next proposition. 

\begin{prop}[Properties of $\bW$] \label{p:bW}
Let $\bmu=(\mu^+,\mu^-), \bnu=(\nu^+,\nu^-) \in \mathcal{P}(\T^d\times\{\pm 1\})$ be such that $\mu^\pm (\T^d) = \nu^\pm (\T^d)$. Then
\begin{enumerate}[label=(\roman{*})]
  \item \label{p:bW:sum} $\bW^2 \big( \bmu, \bnu \big) = W^2 \big( \mu^+, \nu^+ \big) + W^2 \big( \mu^-, \nu^- \big)$, where $W$ is the standard 2-Wasserstein distance on $\mathcal{M}_+(\T^d)$ with cost $\| \cdot \|^2_{\Td}$;
  \item \label{p:bW:gamma} There exist $\gamma \in \Gamma_\circ (\bmu,\bnu)$ and $\gamma^\pm \in \Gamma_\circ (\mu^\pm,\nu^\pm)$ such that $\gamma = (\gamma^+, \gamma^-)$, where $\Gamma_\circ (\mu^\pm,\nu^\pm)$ denotes the set of optimal transport plans for $W$.
\end{enumerate}
\end{prop}

\begin{proof}
To prove \ref{p:bW:sum}, just note that the inequality $\bW^2 \big( \bmu, \bnu \big) \leq W^2 \big( \mu^+, \nu^+ \big) + W^2 \big( \mu^-, \nu^- \big)$ is trivial, since the infimum in the definition of $\bW$ is computed on a larger set. The opposite inequality follows by observing that it is more convenient to redistribute mass within $\T^d\times \{1\}$ and $\T^d\times \{-1\}$ rather than moving mass between the two, since $\diam(\T^d)<2$. Property \ref{p:bW:gamma} follows immediately from \ref{p:bW:sum}.
\end{proof}

Next we extend \cite[Theorem 8.4.7]{AmbrosioGigliSavare08} to Wasserstein spaces on the torus.

\begin{lem}[Derivative of $W$ along curves in $m \mathcal{P}(\T^d)$] \label{lem:dWdt}
Let $m>0$ and let $v \in C([0,T] \times \T^d; \R^d)$ be a vector field such that $x\mapsto v(t,x)$ is Lipschitz continuous in $\T^d$ uniformly in $t$. Let $\mu : (0,T) \to m \mathcal{P}(\T^d)$ be any curve satisfying
\begin{equation}\label{evo:muv}
  \partial_t \mu + \div (\mu v)=0
  \quad \text{in } \mathcal D'((0,T) \times \T^d).
\end{equation}
Then for every $\sigma \in m\mathcal{P}(\T^d)$ it holds that
\begin{equation*}
  \frac{d}{dt} W^2(\mu (t), \sigma) = 2\int_{\T^d \times \T^d} \big( x - y + k(x,y) \big) \cdot v(t, x) d\gamma(x,y)
  \quad \text{for every } 0 < t < T, 
\end{equation*}
where $k(x,y)\in \Z^d$ is such that $\|x-y\|_{\Td} = \|x-y +k(x,y)\|$, and $\gamma \in \Gamma_\circ (\mu (t), \sigma)$.
\end{lem}

\begin{proof}
Let $0<t<T$ be fixed, let $\mu$ be a given solution of \eqref{evo:muv}, and let $h\in \R$ be such that $0 < t+h < T$. We prove that
\begin{equation} \label{p:dWdt:est}
  W^2(\mu (t + h), \sigma) - W^2(\mu (t), \sigma) 
  \leq 2 h \int_{\T^d \times \T^d} \big( x - y + k(x,y) \big) \cdot v(t, x) d\gamma(x,y) + o(|h|)
\end{equation}
as $h \to 0$, where $k(x,y)\in \Z^d$ is as in \eqref{norm:T} with respect to $x$ and $y$. The claim then follows by dividing \eqref{p:dWdt:est} by $h$ and letting $h\to 0$, for both $h>0$ and $h<0$.

For $0 < s < T-t$, let $T_s$ be the map which to any $x \in \T^d$ associates $T_s x = y(t + s)$, where $y$ is the solution of
\begin{equation*} 
\left\{ \begin{aligned}
   y'(s) &= v(s, y(s))
  && 0 < s < T-t, \\
  y(t) &= x.
  &&
\end{aligned} \right.
\end{equation*}
We note that $\mu(t+h) = (T_{h})_\# (\mu(t))$, and since $v$ is continuous in both variables, 
\begin{equation} \label{p:Thx:Tay}
  T_h x = x + h v(t,x) + o(|h|).
\end{equation}

We now prove \eqref{p:dWdt:est}. Let $\gamma \in \Gamma_\circ ( \mu (t), \sigma )$ and set $\gamma_h := (T_h \times \text{id})_\# \gamma \in \Gamma ( \mu (t+h), \sigma )$. Then
\begin{align*}
  W^2(\mu (t + h), \sigma) - W^2(\mu (t), \sigma) 
  &\leq \int_{\T^d \times \T^d} \| x - y \|_{\T^d}^2 d (\gamma_h - \gamma) (x,y) \\
  &= \int_{\T^d \times \T^d} \big( \| T_h x - y \|_{\T^d}^2 - \| x - y \|_{\T^d}^2 \big) d \gamma (x,y) \\
  &\leq \int_{\T^d \times \T^d} \big( \| T_h x - y + k(x,y) \|^2 - \| x - y + k(x,y)\|^2 \big) d \gamma (x,y).
\end{align*}
By expanding the squares of the Euclidean distances and substituting \eqref{p:Thx:Tay} in the estimate above, we get the claim \eqref{p:dWdt:est}.
\end{proof}

We now state an approximation result that will be used in the construction of approximate initial data for the evolution. For the definition of the Orlicz space $L \log L (\mathbb T^d)$ we refer to Appendix~\ref{sect:Orlicz}.

\begin{lem}[Approximation in $\bW$]\label{lemm:approx}
Let $\brho=(\rho^+,\rho^-) \in \big(L \log L (\mathbb T^d)\big)^2\cap \mathcal{P}(\T^d\times\{\pm1\})$. There exists a sequences $\brho_n=(\rho^+_n,\rho^-_n) \in \big(L \log L (\mathbb T^d)\big)^2\cap \mathcal{P}(\T^d\times\{\pm1\})$ with $n\rho^{\pm}_n(\T^d)\in \mathbb{N}$ and $\Ent(\brho_n) \leq c$, uniformly in $n$, such that
\begin{equation*}
\bW^2(\brho, \brho_n) \leq \frac{C}{n}.
\end{equation*}
Moreover, there exist $(x^n,b^n)$ in $\T^d\times \{\pm1\}$ such that the corresponding empirical measures $\bmu_n=(\mu^+_n,\mu^-_n) \in \mathcal{P}(\T^d\times\{\pm1\})$, defined as in \eqref{fd:mun}, satisfy
\begin{equation*}
W^2(\rho^\pm_n, \mu^\pm_n) \leq \frac C{n^{1/d}}.
\end{equation*}
\end{lem}

\begin{proof} We only give a sketch. We construct $\brho_n$ from $\brho$ by moving mass of at most $1/n$ from $\T^d\times \{+1\}$ to $\T^d\times \{-1\}$. By \eqref{d2}, the cost of this transport is at most $((\diam \T^d/2)^2 + 2)/n$, i.e.,
$$
\bW^2(\brho, \brho_n) \leq \frac{C}n.
$$

To construct $\bmu_n$, we choose $\mu_n^\pm(\T^d) = \rho_n^\pm(\T^d)$, so that, by Proposition \ref{p:bW}, it is enough to estimate $W^2(\mu_n^\pm,\rho_n^\pm)$. The idea is to split $\T^d$ into a $d$-cubic grid with cells of size of order $n^{-1/d}$. Then, sequentially for each cell, we move at most $1/n$ mass to the next neighbouring cell such that each cell has mass with a value in $\frac 1n \N$. For the so-constructed $\tilde \rho_n^\pm$, we construct $\mu_n^\pm$ by placing in each of the cells $n \tilde \rho_n^\pm (\text{cell})$ particles (the precise location does not matter). An elementary computation yields that 
\begin{align*}
W(\rho^\pm_n, \mu^\pm_n) 
\leq  W(\rho^\pm_n, \tilde \rho_n^\pm) + W(\tilde \rho_n^\pm, \mu^\pm_n)
\leq \frac C{n^{1/d}}.
\end{align*}
\end{proof}

\medskip

Finally, in the next result we prove a bound for the convolution with $\nabla V_\delta$ that will be crucial in establishing a priori bounds for solutions of the evolutionary problems.

\begin{lem} \label{l:V:props}
Let $\delta>0$, and let $V_\delta: \T^d \to \R$ be a regularised potential satisfying conditions (V1) and (V2). Then there exists a constant $c > 0$, independent of $\delta$, such that for all $f \in L^2 (\mathbb T^d)$ we have 
\begin{align}
 -\int_{\mathbb T^d} (\Delta V_\delta * f) f dx \geq c \,\| \nabla V_\delta * f \|_{H^{d/2} (\mathbb T^d)}^2.
  \label{est:gradVdelta-times-f}
\end{align}
\end{lem}

\begin{proof}
Let $f \in L^2 (\mathbb T^d)$. Since $\Delta V_\delta * f \in L^2 (\mathbb T^d)$, by the Parseval's Theorem and the Convolution Theorem we have
\begin{align*}
 -\int_{\mathbb T^d} (\Delta V_\delta * f) f dx   &= -\sum_{k \in \Z^d} \Big[ \widehat{\Delta V_\delta * f} \big]_k \overline{\widehat f_k} 
 = \sum_{k \in \Z^d} 4 \pi^2 |k|^2  [\widehat V_\delta]_k  \big| \widehat f_k \big|^2\\
 & \geq c  \Big( \sup_{k \in \Z^d\setminus \{0\}}(1 + |k|^2)^{\frac d2} [\widehat V_\delta]_k \Big)\sum_{k \in \Z^d} 4 \pi^2 |k|^2  [\widehat V_\delta]_k  \big| \widehat f_k \big|^2\\
&\geq c \sum_{k \in \Z^d} (1 + |k|^2)^{\frac d2} \big(2\pi |k| [\widehat V_\delta]_k\big)^2 \big| \widehat f_k \big|^2 \\
& \geq c  \| \nabla V_\delta * f \|_{H^{d/2} (\mathbb T^d)}^2,
\end{align*}
where we have also used  (V1) and (V2).
  \qedhere
\end{proof}

\begin{rem}\label{rem:11}
For slip-confined evolutions we will use the following anisotropic variant of \eqref{est:gradVdelta-times-f}, namely
\begin{align*}
 -\int_{\mathbb T^d} (\partial_1^2 V_\delta * f) f dx \geq c \,\| \partial_1 V_\delta * f \|_{H^{d/2} (\mathbb T^d)}^2,
\end{align*}
where $\partial_1$ is shorthand for $\partial_{x_1}$, which can be proved in the same way as \eqref{est:gradVdelta-times-f}.
\end{rem}

\subsection{Detailed formulation of the two problems} 
\label{ss:cases:2split}
In this section we describe the isotropic and anisotropic evolutionary problems separately.

\subsubsection{Case 1: Isotropic evolution in dimension $d\geq 1$}
\label{subsec:description-isotropic}

This is the case that we discussed in the introduction for $d=2$.

The difference between the isotropic case considered in this section and the anisotropic case of Section \ref{subsec:description-confined} is in the evolution. 
In the isotropic case the evolution equation is
\begin{equation} \label{for:evo:dyncs:edge:2D-repeated}
     \frac{d x}{d t} (t) = - n \nabla  E_n (x(t); b)  \quad \text{on } (0,T]
\end{equation}
as introduced in~\eqref{for:evo:dyncs:edge:2D}, which in components and by \eqref{eqn:disc:energy-repeated} reads as
\begin{equation} \label{pf:xit}
     \frac{d x_i}{d t} (t)
     = - \frac1n \sum_{j = 1}^n b_i b_j \nabla V_\delta (x_i(t) - x_j(t))
     - b_i \nabla U(x_i(t)),
     \quad \text{on } (0,T], \ i = 1,\ldots,n.
\end{equation}

In the two-dimensional case $d=2$, this model is inspired by the study of the motion of straight and parallel dislocations, which can be represented by points in the plane. 
In particular, the evolution \eqref{for:evo:dyncs:edge:2D-repeated} (or \eqref{pf:xit}) is best compared to the evolution of \emph{screw} dislocations, which are free to move in multiple directions. The isotropy assumption means that dislocations are allowed to move in every direction (see e.g.~\cite{BlassFonsecaLeoniMorandotti15,BonaschiVanMeursMorandotti16, AlicandroDeLucaGarroniPonsiglione17} for a more faithful treatment of preferred slip directions). When considered in the full space $\R^2$, screw dislocations generate an interaction potential~$V$ that is explicit: $V(x) = -\log |x|$ (up to material constants, which we disregard in this paper). 

As the first step in establishing evolutionary convergence, we rewrite the evolutions of the particles in \eqref{for:evo:dyncs:edge:2D-repeated} in terms of the empirical measures $\mu_n^\pm$, defined as in \eqref{fd:mun}, associated to their positions. This is convenient since the Groma-Balogh equations \eqref{f:GBintro} that we want to obtain in the many-particle limit $n\to\infty$ of \eqref{for:evo:dyncs:edge:2D-repeated} are given in terms of measures too. In fact, for any fixed $\delta>0$, the empirical measures $\mu_n^\pm$ are solutions of the regularised Groma-Balogh equations
\begin{equation} \label{f:GB-regularized-repeated}
\begin{aligned}
&\partial_t \rho^+ 
= &\div \big( \rho^+ ( \nabla V_{\delta} * (\rho^+ - \rho^-) + \nabla U ) \big),\\
&\partial_t \rho^- 
= &-\div \big( \rho^- ( \nabla V_{\delta} * (\rho^+ - \rho^-) + \nabla U) \big).
\end{aligned}
\end{equation} 
The convergence result, Theorem~\ref{t:evoConv}, states that, for $\delta=\delta_n\to 0$ as $n\to \infty$ sufficiently slowly, $\mu^\pm_n$ converge to solutions of the unregularised version
\begin{equation} \label{f:GB-repeated}
\begin{aligned}
\partial_t \rho^+ 
&=& \div \big( \rho^+ ( \nabla V * (\rho^+ - \rho^-) + \nabla U ) \big),\\
\partial_t \rho^- 
&=& - \div \big( \rho^- ( \nabla V * (\rho^+ - \rho^-) + \nabla U) \big).
\end{aligned}
\end{equation}
Definition~\ref{def:GB-sol-isotropic} below specifies the solution concept for equations \eqref{f:GB-repeated}.

\begin{rem}[Applications of $d=1$]
At least two model situations motivate the case $d=1$, namely (a) dislocations in two dimensions that are confined to a single slip plane, and (b) periodic walls of edge dislocations~\cite{ScardiaPeerlingsPeletierGeers14}. In both cases the potential scales as $-\log|x|$ for small $|x|$. The positivity of the Fourier transform of the potential (b) follows by specialising the discussion on $V_{\textrm{wall}}$ below to the one-dimensional case.
\end{rem}

\subsubsection{Case 2: Two-dimensional domain, slip-plane-confined motion}
\label{subsec:description-confined}

The case of \emph{edge} dislocations naturally leads us to consider the case of a two-dimensional domain where the dislocations are confined to fixed slip directions, which we take to be horizontal (i.e.,\ parallel to the first coordinate axis in $\T^2$). Edge dislocations generate a different interaction potential than screw dislocations. In the case of horizontal slip, the interaction potential in $\R^2$ is
\begin{equation*}
V_{\rm{edge}}(x) = -\log |x| + \frac{(x\cdot e_1)^2}{|x|^2},
\end{equation*} 
see, e.g., (5-16) in \cite{HirthLothe82}. We construct the potential\footnote{While \cite[Lem.~2.1]{CannoneEl-HajjMonneauRibaud10}  only defines the second derivative $\partial_{11} V$, our setting requires the potential $V$ itself.} $V$ on $\T^2$ by carefully summing over $\Z^2$ shifted copies of $V_{\rm{edge}}$ first in the vertical and then in the horizontal direction. In \cite[App.~A]{VanMeurs15TH} it is shown that the pointwise limit
\begin{equation}
\label{def:Vwall-VM-limit}
  V_{\text{wall}}(x) := \lim_{M \to \infty} V_M (x) := \lim_{M \to \infty} \Big( C_M + \sum_{m = -M}^M V_{\rm{edge}} (x + m e_2) \Big),
  \quad C_M := 2 \sum_{m = 1}^M \log m,
\end{equation}
is well-defined for all $x \in \R^2$, except at the singular points $(0,k)$, $k\in\Z$, of any shifted $V_{\rm{edge}}$, and that $V_{\text{wall}}$ is vertically $1$-periodic with exponentially decaying tails in the horizontal direction. A straightforward computation shows that $( V_M )_{M \in \N}$ converges uniformly in $\R^2$ away from the singularities; since the singular behaviour at each singularity is given by a translated copy of $V_{\mathrm{edge}}\in L^1_{\mathrm{loc}}(\R^2)$,  we find that the convergence in~\eqref{def:Vwall-VM-limit} takes place in $L^1_{\text{loc}} (\R^2)$.

Together with the exponential tails, we then find that the following limit is well-defined in $L^1_{\mathrm{loc}} (\R^2)$:
\begin{equation*}
  V(x) := \lim_{N \to \infty} \sum_{n = -N}^N V_{\text{wall}} (x + n e_1).
\end{equation*}
It is clear that $V$ is $\Z^2$-periodic. Moreover, since $V_{\rm{edge}}$ satisfies 
\begin{equation}\label{VR2eqn}
-\int_{\R^2}\Delta^2 V_{\rm{edge}} \varphi = 4\pi \partial^2_{2}\varphi(0)
\quad \forall \varphi\in \mathcal D(\R^2),
\end{equation} 
in particular all finite sums satisfy \eqref{VR2eqn} for  $\varphi\in \mathcal D((-1,1)^2)$, and then
\begin{equation}\label{VT2eqnOnR2}
-\int_{\R^2}\Delta^2 V \varphi = 4\pi \partial^2_{2}\varphi(0)
\quad \forall \varphi\in \mathcal D((-1,1)^2)\,.
\end{equation} 
Therefore, using the periodicity of $V$, it is easy to see that
\begin{equation}\label{VT2eqn}
-\int_{\T^2}\Delta^2 V\varphi = 4\pi \partial^2_{2}\varphi(0)
\quad \forall \, \varphi\in C^\infty(\R^2) \text{ with $\varphi$ } \Z^2\hbox{-periodic.}
\end{equation} 
In Fourier space, equation \eqref{VT2eqn} becomes
\begin{equation*}
-(4\pi^2|k|^2)^2 \widehat V_k = 4\pi (-4\pi^2) k_2^2,
\end{equation*} 
and hence for every $k\neq 0$,
\begin{equation*}
 \widehat V_k = \frac{1}\pi \frac{k_2^2}{|k|^4} \geq 0.
\end{equation*}

The anisotropic and regularised evolution satisfied by the dislocation positions is 
\begin{equation} \label{for:evo:dyncs:edge:2D-confined}
     \frac{d x}{d t} (t) = - n \nabla_1  E_n (x(t); b),
\end{equation}
where the energy $E_n$ is as in \eqref{eqn:disc:energy-repeated}, and the \textit{confined} gradient of a function $f:\Omega\subset\R^{2n}\to \R$ is defined in components, for $i=1,\dots,2n$, as
\begin{equation*}
(\nabla_1 f)_i:= 
\begin{cases}
\medskip
\dfrac{\partial f}{\partial x_i} \quad &\text{if } i = 2k-1, \quad k=1,\dots,n,\\
0 \quad &\text{if } i = 2k, \qquad \ \ k=1,\dots,n.
\end{cases}
\end{equation*}
Hence in \eqref{for:evo:dyncs:edge:2D-confined} it is only the horizontal component of the positions that varies in time, namely $x^2(t) = x^2(0)$ for every $t$, where $x^2 = (x^2_1, \dots, x_n^2)$, and $x_i^2 = x_i\cdot e_2$, for every $x_i \in \T^2$.

The empirical measures $\mu_n^\pm$ defined in \eqref{fd:mun} and associated to the solution of \eqref{for:evo:dyncs:edge:2D-confined} satisfy the regularised and constrained Groma-Balogh equations
\begin{equation} \label{f:GB-regularized-confined}
\begin{aligned}
\partial_t \rho^+ 
&=& &\partial_1 \big(\rho^+ (\partial_1 V_{\delta} * (\rho^+ - \rho^-) + \partial_1 U ) \big),\\
\partial_t \rho^- 
&=& - &\partial_1 \big(\rho^- (\partial_1 V_{\delta} * (\rho^+ - \rho^-) + \partial_1 U) \big).
\end{aligned}
\end{equation}
The convergence result, Theorem \ref{t:evoConv:conf}, states that, if $\delta=\delta_n$ tends to zero as $n\to \infty$ sufficiently slowly, $\mu^\pm_n$ converge to solutions of the corresponding unregularised version, namely
\begin{equation} \label{f:GB-slip-plane-confined}
\begin{aligned}
\partial_t \rho^+ 
&=& &\partial_1\big( \rho^+ (\partial_1 V * (\rho^+ - \rho^-) + \partial_1 U ) \big),\\
\partial_t \rho^- 
&=& - &\partial_1 \big(\rho^- (\partial_1 V * (\rho^+ - \rho^-) + \partial_1 U) \big).
\end{aligned}
\end{equation}

%%%%%%%%%%%%%%%%%%%%%%%%%%%%%%%%%%%%%%%%%%%%%%%

\section{Convergence results}
\label{sec:convergence}
In this section we prove that the solutions of the discrete gradient-flow equations \eqref{for:evo:dyncs:edge:2D-repeated} (or \eqref{for:evo:dyncs:edge:2D-confined}) associated to the regularised energy~\eqref{eqn:disc:energy} converge to limit measures that solve equations~\eqref{f:GB-repeated} (or \eqref{f:GB-slip-plane-confined}) provided that $\delta_n\to0$ sufficiently slowly.  
\smallskip

We consider the two cases outlined in Section \ref{ss:cases:2split}: isotropic in dimension $d\geq 1$ and slip-confined in two dimensions.
In both cases our approach builds on the existence proof given by Cannone, El Hajj, Monneau, and Ribaud in~\cite{CannoneEl-HajjMonneauRibaud10}, where they introduce a functional framework in which the Groma-Balogh equations~\eqref{f:GB-slip-plane-confined} for the slip-plane-confined case with $U=0$ have a meaningful weak solution. The challenge they face is to give the nonlinear terms $\rho^\pm \partial_1 V*(\rho^+-\rho^-)$ a meaning in the sense of distributions on $(0,T)\times \T^2$, for some $T>0$: since $\rho^+$ and $\rho^-$ are \emph{a priori} only measures, and $\partial_1 V$ is singular, this product need not have any meaning unless we make stronger assumptions on $\rho^\pm$. 
This is done by means of the following lemmas, which have been proved in \cite[Propositions~1.3]{CannoneEl-HajjMonneauRibaud10} in the case $d=2$, and is a key result in their analysis.

\begin{lem}
\label{l:product-term} Let $d\geq 1$, $T>0$, $f\in L^1(0,T; H^{\frac d2}(\T^d))$ and $g\in L^\infty(0,T;L\log L(\T^d))$. Then 
\[
fg\in L^1((0,T)\times \T^d).
\]
\end{lem}

For the proof of this lemma for general $d\geq 1$ we use part \ref{l:O:Hol:L1} of Lemma \ref{l:Orlicz} and integrate in time. Then we conclude by the embedding of $H^{\frac d2}(\T^d)$ into $\Exp(\T^d)$ in Lemma \ref{l:Orlicz} \ref{l:O:H1:Exp:cp}.
\smallskip

In the next sections we describe and extend the functional framework of~\cite{CannoneEl-HajjMonneauRibaud10} for each of the two cases we consider, supplement it with new, quantitative discrete-to-continuum estimates, and formulate and prove our convergence results.

\subsection{Case 1: Isotropic drag law in dimension $d\geq 1$.} Here we consider the isotropic case introduced in Sections~\ref{sec:Introduction} and~\ref{subsec:description-isotropic}.

\smallskip

\noindent  
We first discuss the different concepts of solutions that we use. For a regularised potential $V_{\delta}$ the function $\nabla V_{\delta}*(\rho^+-\rho^-)$ is  smooth; therefore the equations~\eqref{f:GB-regularized-repeated} are a pair of convection equations with smooth velocity fields. In particular the equations preserve the regularity of the initial data, and on finite time intervals no derivatives blow up. Therefore the concept of a solution poses no problems, and one can consider both smooth and measure-valued solutions of these equations. The empirical measures $\mu_n^\pm$ constructed above from the discrete solutions of \eqref{for:evo:dyncs:edge:2D-repeated} are examples of such measure-valued solutions.

For the case of a singular potential $V$, we need to be more careful. We use Lemma \ref{l:product-term} above to guarantee that the integrals involving $\rho^\pm \nabla V*(\rho^+-\rho^-)$ are meaningful, and we define the concept of solution as follows.
\begin{defn}
\label{def:GB-sol-isotropic}
Let $d\geq 1$, let $\brho_\circ =(\rho^+_\circ, \rho^-_\circ)\in \big( L \log L (\mathbb T^d) \big)^2 \cap\mathcal P(\T^d\times \{\pm1\})$, let $T>0$, and let $\brho=(\rho^+, \rho^-): [0,T]\to \mathcal P(\mathbb T^d\times \{\pm1\})$. We say that $\brho$ is a solution of~\eqref{f:GB-repeated} in $[0,T]$, with initial datum $\brho_\circ$, if 
\begin{itemize}
\smallskip
\item[$(1)$] $\rho^+, \rho^-\in L^\infty(0,T;L\log L(\T^d))$ and $\nabla V*\kappa\in L^1(0,T;H^{\frac d2}(\T^d;\R^d))$, where $\kappa = \rho^+-\rho^-$;
\item[$(2)$] For every $\varphi, \psi \in C_c^\infty ([0,T)\times \T^d)$
\begin{subequations}
\begin{align*}
&\int_0^T\int_{\T^d} \rho^+ (t,x) \big(\partial_t \varphi(t,x) - \nabla \varphi(t,x) \cdot  \nabla (V*\kappa + U)(t,x) \big) dx dt + \int_{\T^d} \rho_\circ^+(x) \varphi(0,x) dx =0,\\
&\int_0^T\int_{\T^d} \rho^- (t,x) \big(\partial_t \psi(t,x) + \nabla \psi(t,x) \cdot  \nabla (V*\kappa + U)(t,x) \big) dx dt + \int_{\T^d} \rho_\circ^-(x) \psi(0,x) dx =0.
\end{align*}
\end{subequations}
\end{itemize}
\end{defn}

\bigskip
The \emph{existence} of a solution in the sense of Definition~\ref{def:GB-sol-isotropic} will follow from our convergence Theorem~\ref{t:evoConv} below. \emph{Uniqueness}  of solutions  for such weak solutions is currently open; for stronger solution concepts, under assumptions of higher regularity, various uniqueness proofs have been constructed for related systems~\cite{ElHajj10,Mainini12a,LiMiaoXue14}.

\medskip

\noindent
The main theorem of this section is the following. We set $\lambda_{\delta} := \|D^2 V_{\delta}\|_{L^\infty(\T^d)}$ and note that $\lambda_{\delta}\to\infty$ as $\delta\to0$. 

\begin{thm}[Evolutionary convergence] \label{t:evoConv} 
Let $d\geq 1$, let $\brho_\circ \in \big( L \log L (\mathbb T^d) \big)^2 \cap \mathcal P (\mathbb T^d \times \{\pm1\})$, and let $(\brho_{\circ,n}), (\bmu_{\circ,n})\subset \mathcal P (\mathbb T^d \times \{\pm1\})$ be approximating sequences for $\brho_\circ$ as in Lemma \ref{lemm:approx}. Let $T>0$ be fixed, and let $\delta_n>0$ be such that
\begin{equation}
\label{cond:evoConv-conv-initial}
\exp (3 \lambda_{\delta_n} T) \, \bW(\bmu_{\circ,n}, \brho_{\circ,n}) \to 0, \qquad \text{as }n \to \infty.
\end{equation}
Let $t\mapsto \bmu_n(t)$ be the empirical measure of the solution $t\mapsto x^n(t)$ of~\eqref{for:evo:dyncs:edge:2D-repeated} on $[0,T]$, with $\delta=\delta_n$ and initial datum $\bmu_{\circ,n}$.

Then there exists the limit $\brho(t):=\lim_{n\to \infty}\bmu_n(t)$ with respect to $\bW$, uniformly in $t \in [0,T]$, where $\brho$ is a solution of the unregularised Groma-Balogh equations~\eqref{f:GB-repeated} in the sense of Definition~\ref{def:GB-sol-isotropic} with initial datum $\brho_\circ$.
\end{thm}

\begin{rem} 
While \eqref{cond:evoConv-conv-initial} suggests that the choice of $\delta_n$ depends on the initial data $(\bmu_{\circ,n})$, we can also reverse the dependence. Indeed, requiring instead that
\begin{equation*}
\exp (3 \lambda_{\delta_n} T) \, n^{-1/d} \to 0 \qquad \text{as }n \to \infty,
\end{equation*}
the proof of Lemma \ref{lemm:approx} provides for every $\brho_\circ \in \big( L \log L (\mathbb T^d) \big)^2 \cap \mathcal P (\mathbb T^d \times \{\pm1\})$ a class of approximating sequences $(\bmu_{\circ,n})_n$ such that \eqref{cond:evoConv-conv-initial} holds for each such sequence.
\end{rem}

\noindent
\textit{Structure of the proof.} 
We prove Theorem \ref{t:evoConv} by constructing an `intermediate' trajectory $\brho_{\delta_n}$: $\brho_{\delta_n}$ is obtained by solving the regularised equations~\eqref{f:GB-regularized-repeated} with the modified initial datum $\brho_{\circ,n}$. We then show that $\bW (\brho_{\delta_n}, \bmu_n)$ vanishes as $n$ tends to $\infty$ (as a consequence of Lemma~\ref{l:dist-bmu-brho_n}) and that $(\brho_{\delta_n})$ converges to a solution of~\eqref{f:GB-repeated} (Theorem~\ref{thm:delta_to_0}). 

\begin{lem}[Existence, uniqueness and Gronwall estimate for \eqref{f:GB-regularized-repeated}]
\label{l:dist-bmu-brho_n} Let $d\geq 1$, and let $\delta, T>0$ be fixed. Then for every $\bmu_\circ \in \mathcal P (\mathbb T^d \times \{\pm1\})$ there exists a solution $\bmu_\delta\in C([0,T]; \mathcal P (\mathbb T^d \times \{\pm1\}))$ of the regularised equations~\eqref{f:GB-regularized-repeated} with initial datum $\bmu_{\circ}$ (where continuity is intended with respect to $\bW$). Moreover, if $\bmu_\circ, \bnu_\circ \in \mathcal P (\mathbb T^d \times \{\pm1\})$ satisfy $\mu^\pm_\circ(\T^d) = \nu^\pm_\circ(\T^d)$, then any solutions $\bmu_\delta, \bnu_\delta \in C([0,T]; \mathcal P (\mathbb T^d \times \{\pm1\}))$ to \eqref{f:GB-regularized-repeated} with initial data $\bmu_\circ$ and $\bnu_\circ$ respectively, satisfy
\begin{equation} \label{f:Gronwall00}
  \bW (\bmu_\delta (t), \bnu_{\delta} (t))
  \leq c\, e^{3\lambda_{\delta} t} \bW (\bmu_{\circ} , \bnu_{\circ} )
  \quad \text{for all }  t\in [0,T],
\end{equation}
where $c$ is a constant independent of $\delta$ and $t$. In particular, if $\bmu_0=\bnu_0$, then $\bmu_\delta(t) = \bnu_\delta(t)$ for every $t\in [0,T]$.
\end{lem}

\begin{proof} We split the proof into two steps.
\smallskip

\textit{Step 1: Existence.} Let $\bmu_{\circ}\in \mathcal P (\mathbb T^d \times \{\pm1\})$. The existence of a solution for the regularised equations~\eqref{f:GB-regularized-repeated} with initial datum $\bmu_{\circ}\in \mathcal P (\mathbb T^d \times \{\pm1\})$ follows by discrete (in space) approximation. More precisely, from Lemma \ref{lemm:approx} we get $(x^n_\circ, b^n) \in (\T^d)^n\times \{\pm 1\}^n$ such that the corresponding empirical measures $\bmu_{\circ,n}$ converge to $\bmu_\circ$ in $\bW$ as $n\to \infty$. 
Let $x^n(t)$, for $t\in [0,T]$, denote the solution of \eqref{pf:xit} with initial datum $x(0) = x^n_\circ$, and let $\bmu_n(t)$ denote the corresponding empirical measure. Note that the sequence $\bmu_n: [0,T] \to \mathcal P (\mathbb T^d \times \{\pm1\})$ satisfies all the assumptions of the refined version of the Ascoli-Arzel\`a Theorem \cite[Proposition 3.3.1]{AmbrosioGigliSavare08}. Indeed the metric space $(\mathcal P (\mathbb T^d \times \{\pm1\}), \bW)$ is complete, and $\mathcal P (\mathbb T^d \times \{\pm1\})$ is sequentially compact with respect to $\bW$ (note that the domain $\T^d$ is bounded, so the convergence in measure and in $\bW$ are equivalent). Moreover, the map $t \mapsto \bmu_n(t)$ is Lipschitz continuous with respect to $\bW$, with Lipschitz constant independent of $n$ (but depending on $\delta$). To see this first note that, from the equations~\eqref{f:GB-regularized-repeated} satisfied by $\mu_n^\pm$ we have that $\mu_n^\pm(s)(\T^d)=\mu_n^\pm(t)(\T^d)$ for every $s\leq t$, hence, by Proposition \ref{p:bW}, it is sufficient to show that $t\mapsto \mu_n^\pm(t)$ are Lipschitz continuous with respect to $W$. The latter follows since for every $s\leq t$ we have 
\begin{align*}
W^2(\mu_n^\pm(s),\mu_n^\pm(t)) &= \frac1n \inf_{\sigma\in \mathcal{S}_n} \sum_{\substack{{i=1}\\{b_i=\pm1}}}^n\|x_i(s) - x_{\sigma(i)}(t)\|^2_{\Td} \leq 
\frac1n \sum_{\substack{{i=1}\\{b_i=\pm1}}}^n\|x_i(s) - x_i(t)\|^2_{\Td}\\
&\leq \left(\big(\|\nabla V_\delta\|_{L^\infty(\Td)} + \|\nabla U\|_{L^\infty(\Td)}\big) |t-s| \right)^2,
\end{align*}
where we used (V4), and $\mathcal{S}_n$ denotes the set of permutations of $\{1,\dots,n\}$.

By the refined version of the Ascoli-Arzel\`a Theorem we conclude that there exists a time-independent subsequence of $(\bmu_n)$ (not relabelled) and a limit curve $\bmu_\delta = (\mu^+_\delta, \mu^-_\delta): [0,T] \to  \mathcal P (\mathbb T^d \times \{\pm1\})$ such that 
\begin{equation}\label{e:Wnd}
\bW(\bmu_n(t),\bmu_\delta(t)) \to 0 \quad \textrm{for every } \ t\in [0,T],
\end{equation}
and $\bmu_\delta$ is $\bW$-continuous in $[0,T]$. 

We show that $\bmu_\delta$ is a solution of~\eqref{f:GB-regularized-repeated} with initial datum $\bmu_{\circ}$. First of all, $\bmu_n$ is a solution (in the sense of distributions) of~\eqref{f:GB-regularized-repeated} with initial datum $\bmu_{\circ,n}$, namely
\begin{equation}\label{eqn:mun}
\int_0^T\int_{\T^d}\big(\partial_t \varphi(t,x) \mp \nabla \varphi(t,x) \cdot  \nabla (V_\delta*\kappa_n + U)(t,x)\big) d \mu_n^{\pm} (t,x) dt + \int_{\T^d} \varphi(0,x) d \mu_{\circ,n}^{\pm}(x) =0
\end{equation}
for every $\varphi\in C_c^\infty ([0,T)\times \T^d)$, where $\kappa_n=\mu_n^+-\mu_n^-$. Since by construction $\bW(\bmu_{\circ,n},\bmu_\circ) \to 0$, and ${\bW}(\bmu_n(t),\bmu_\delta(t)) \to 0$ for every $t\in [0,T]$ by \eqref{e:Wnd}, we can immediately pass to the limit in the first and last term of \eqref{eqn:mun} (note that convergence in $\bW$ implies $\textrm{weak}^\ast$-convergence of the positive and negative components). For the nonlinear term, we pass to the limit pointwise in $t\in (0,T)$, and conclude afterwards by applying the Dominated Convergence Theorem over the time integral. The spatial integral reads as
\[
  \iint_{(\T^d)^2} \nabla \varphi(x) \cdot \nabla V_\delta (x - y) \, d (\mu_n^+ \otimes \mu_n^{\pm}) (y,x) - \iint_{(\T^d)^2} \nabla \varphi(x) \cdot \nabla V_\delta (x - y) \, d (\mu_n^- \otimes \mu_n^{\pm}) (y,x).
\]
Since $(x,y) \mapsto \nabla \varphi(x) \cdot \nabla V_\delta (x - y)$ is continuous on $(\T^d)^2$, we can directly pass to the limit $n \to \infty$.
\smallskip

\textit{Step 2: Gronwall inequality.}  Let now $\bmu_\delta, \bnu_\delta \in C([0,T]; \mathcal P (\mathbb T^d \times \{\pm1\}))$ be solutions of \eqref{f:GB-regularized-repeated} with initial data $\bmu_\circ, \bnu_\circ \in \mathcal P (\mathbb T^d \times \{\pm1\})$ satisfying $\mu_\circ^\pm(\T^d) = \nu_\circ^\pm(\T^d)$. From the equations \eqref{f:GB-regularized-repeated} we see that $\mu_\delta^\pm(t)(\T^d) = \mu_\circ^\pm(\T^d)$ and $\nu_\delta^\pm(t)(\T^d) = \nu_\circ^\pm(\T^d)$ for every $t\in [0,T]$, so in particular $\mu_\delta^\pm(t)(\T^d) = \nu_\delta^\pm(t)(\T^d)$ for every $t$. By Proposition \ref{p:bW} we have that
\begin{align*}
\bW^2 \big(\bmu_\circ, \bnu_\circ \big) &= W^2 \big( \mu^+_\circ, \nu^+_\circ \big) + W^2 \big( \mu^-_\circ, \nu^-_\circ \big)\\
\bW^2 \big(\bmu_\delta(t), \bnu_\delta(t)) &= W^2 \big( \mu^+_\delta(t), \nu^+_\delta(t) \big) + W^2 \big( \mu^-_\delta(t), \nu^-_\delta(t) \big).
\end{align*}

Since the curves $\mu_\delta^\pm, \nu^\pm_\delta$ are continuous in time, we have by $(V4)$ that the fluxes $v^\pm[\bmu_\delta]$ and $v^\pm[\bnu_\delta]$ from the continuity equations \eqref{f:GB-regularized-repeated}, defined by 
$$
v^-[\bmu]:= \nabla V_\delta\ast (\mu^+-\mu^-) + \nabla U, \quad v^+[\bmu] = - v^-[\bmu], \quad  \bmu=(\mu^+,\mu^-) \in \P (\T^d \times \{\pm 1\}),
$$
are continuous in time. They are moreover Lipschitz continuous in space, since (omitting $t$)
\begin{align}\label{e:fieldsLip}
|{v}^\pm[\bmu_\delta](x) - {v}^\pm[\bmu_\delta](y)| 
\leq  (\lambda_\delta + \|D^2U\|_{L^\infty(\T^d)})\ \|x-y\|_{\T^d}. 
\end{align}
Hence, Lemma \ref{lem:dWdt} applies. This yields that for every $\bsigma \in \P (\T^d \times \{\pm 1\})$ with $\sigma^\pm (\T^d) = \mu_\circ^\pm(\T^d)$,
\begin{align}\label{one-fixed-1}
\frac{d}{dt} W^2(\mu^\pm_\delta(t), \sigma^\pm) &= 2\int_{\T^d \times \T^d} \big( x - y + k(x,y) \big) \cdot{v}^\pm[\bmu_\delta](x,t) d\gamma_1^\pm(x,y),\\\label{one-fixed-2}
\frac{d}{dt} W^2(\nu^\pm_\delta(t), \sigma^\pm) &= 2\int_{\T^d \times \T^d} \big( x - y + k(x,y) \big) \cdot{v}^\pm[\bnu_\delta](x,t) d\gamma_2^\pm(x,y),
\end{align} 
where $\gamma_1^\pm \in \Gamma_\circ(\mu_\delta^\pm, \sigma^\pm)$ and $\gamma_2^\pm \in \Gamma_\circ(\nu_\delta^\pm, \sigma^\pm)$. 
Combining \eqref{one-fixed-1} and \eqref{one-fixed-2} we deduce that 
\begin{align}\label{combined}
\frac{d}{dt} W^2(\mu^\pm_\delta(t), \nu^\pm_\delta(t)) 
\leq 2\int_{\T^d \times \T^d} \big( x - y + k(x,y) \big)\cdot \big( {v}^\pm[\bmu_\delta](x,t) - {v}^\pm[\bnu_\delta](y,t) \big) d\gamma^\pm(x,y), 
\end{align}
where $\gamma^\pm \in \Gamma_\circ(\mu_\delta^\pm(t), \nu_\delta^\pm(t))$ and $k(x,y) = -k(y,x)$ is used. We now observe that (omitting $t$)
\begin{align*}
&|{v}^\pm[\bmu_\delta](x) - {v}^\pm[\bnu_\delta](y)| 
\leq |{v}^\pm[\bmu_\delta](x) - {v}^\pm[\bmu_\delta](y)| + |{v}^\pm[\bmu_\delta](y) - {v}^\pm[\bnu_\delta](y)| \\
&\qquad  \leq (\lambda_\delta + \|D^2U\|_{L^\infty(\T^d)})\ \|x-y\|_{\T^d} + \lambda_\delta (W(\mu^+_\delta(t),\nu_\delta^+(t)) + W(\mu^-_\delta(t),\nu_\delta^-(t))),
\end{align*}
where we have used the estimate
\begin{align*}
|{v}^\pm[\bmu_\delta](y) - {v}^\pm[\bnu_\delta](y)| &= \left|\int_{\T^d} \nabla V_\delta(x-y) d\big((\mu^+_\delta-\mu^-_\delta)-(\nu^+_\delta-\nu^-_\delta)\big)(y) \right| \\
&= \left|\int_{\T^d} \nabla V_\delta(x-y) d\big((\mu_\delta^+-\nu^+_\delta)-(\mu^-_\delta-\nu^-_\delta)\big)(y) \right|\\
&\leq \lambda_\delta \big(W_1(\mu^+_\delta,\nu_\delta^+) + W_1(\mu^-_\delta,\nu_\delta^-)\big)\\
&\leq \lambda_\delta \big(W(\mu^+_\delta,\nu_\delta^+) + W(\mu^-_\delta,\nu_\delta^-)\big),
\end{align*}
and $W_1$ denotes the 1-Wasserstein\footnote{Recall the characterisation of the 1-Wasserstein distance as
$$
W_1(\mu,\nu)=\sup\left\{\int_{\T^d} f d\mu - \int_{\T^d} f d\nu \, : \ f \ 1\hbox{-Lipschitz} \right\}
$$
} distance. Hence, from \eqref{combined} we obtain the bound
\begin{align*}
\frac{d}{dt} W^2(\mu^\pm_\delta(t), \nu^\pm_\delta(t)) &\leq 2(\lambda_\delta + \|D^2U\|_{L^\infty(\T^d)}) \int_{\T^d\times \T^d} \|x- y\|^2_{\T^d} d\gamma^\pm(x,y) \\ 
& + 2  \lambda_\delta  \big(W(\mu^+_\delta,\nu_\delta^+) + W(\mu^-_\delta,\nu_\delta^-)\big) \int_{\T^d\times \T^d}  \|x- y\|_{\T^d} d\gamma^\pm(x,y)\\
&\leq 2(\lambda_\delta + \|D^2U\|_{L^\infty(\T^d)}) W^2(\mu^\pm_\delta(t), \nu^\pm_\delta(t))  \\ 
&+ 2 \lambda_\delta \big(W(\mu^+_\delta,\nu_\delta^+) + W(\mu^-_\delta,\nu_\delta^-)\big) W(\mu^\pm_\delta(t), \nu^\pm_\delta(t)).
\end{align*}
Adding up the estimates for the positive and the negative parts of the measures we conclude that 
\begin{align*}
\frac{d}{dt} \bW^2(\bmu_\delta(t), \bnu_\delta(t))  &\leq 2(\lambda_\delta + \|D^2U\|_{L^\infty(\T^d)}) \bW^2(\bmu_\delta(t), \bnu_\delta(t)) \\
&+  2 \lambda_\delta \big(W(\mu^+_\delta,\nu_\delta^+) + W(\mu^-_\delta,\nu_\delta^-)\big)^2\\
& \leq 2(3\lambda_\delta + \|D^2U\|_{L^\infty(\T^d)}) \bW^2(\bmu_\delta(t), \bnu_\delta(t)).
\end{align*}
By using the Gronwall inequality, we obtain the final estimate 
$$
\bW(\bmu_\delta(t), \bnu_\delta(t)) \leq \exp\big(3\lambda_\delta + \|D^2U\|_{L^\infty(\T^d)}\big) \bW(\bmu_\circ, \bnu_\circ),
$$
which gives \eqref{f:Gronwall00}.
\end{proof}

We now state the convergence of the regularised solution $\brho_{\delta}$, whose existence and uniqueness are established in Lemma \ref{l:dist-bmu-brho_n}, to a solution $\brho$ of~\eqref{f:GB-repeated}. 
\medskip

\begin{thm}\label{thm:delta_to_0} Let $\delta, T>0$ be fixed. 
Let $\brho_\circ, \brho_{\circ,\delta} \in \big( L \log L (\mathbb T^d) \big)^2 \cap \mathcal P (\mathbb T^d \times \{\pm1\})$ be such that $\brho_{\circ,\delta}\xweakto \ast \brho_\circ$ in measure as $\delta\to 0$, and $\Ent(\brho_{\circ,\delta})\leq c$ uniformly in $\delta$. Let $\brho_\delta \in  C([0,T]; \mathcal P (\mathbb T^d \times \{\pm1\}))$ be the solution of the regularised equation~\eqref{f:GB-regularized-repeated} with initial datum $\brho_{\circ,\delta}$.

 Then the limit $\brho:=\lim_{\delta\to 0}\brho_\delta$ exists (up to a subsequence) with respect to the $\textrm{weak}^\ast$ convergence in $L^\infty(0,T; (L\log L (\T^d))^2 )$, and $\brho$ is a solution of the Groma-Balogh equations~\eqref{f:GB-repeated} in the sense of Definition~\ref{def:GB-sol-isotropic} with initial datum $\brho_\circ$.
\end{thm}

In what follows we use the shorthand $L^\infty( L \log L )$ for the space $L^\infty(0,T; L\log L (\T^d) )$, and adopt this convention for any function space on $(0,T) \times \T^d$.

\smallskip

The next lemma provides an a priori estimate for $(\brho_\delta)$ (uniformly in $\delta$), in terms of the entropy defined in \eqref{def:entropy}.

\begin{lem}[A priori estimate for $(\brho_\delta)$]
\label{l:entropy-estimate}
Let $d\geq 1$, and let $\delta, T>0$ be fixed. Let $\brho_{\circ,\delta}\in \big( L \log L (\mathbb T^d) \big)^2 \cap \mathcal P (\mathbb T^d \times \{\pm1\})$, and let $\brho_\delta = (\rho^+_\delta,\rho^-_\delta)\in C([0,T]; \mathcal P (\mathbb T^d \times \{\pm1\}))$ be the solution of \eqref{f:GB-regularized-repeated} given by Lemma~\ref{l:dist-bmu-brho_n} with initial datum $\brho_{\circ,\delta}$. Then $\brho_\delta \in L^\infty(0,T; (L \log L(\T^d))^2 )$, and for all $t\in [0,T]$ we have
\begin{equation}
\label{est:entropy}
 \Ent(\brho_\delta(t)) + c \int_0^t \| \nabla V_\delta * ( \rho^+_\delta-\rho^-_\delta)(\tau) \|_{H^{d/2} (\mathbb T^d)}^2\, d\tau 
 \leq  C + \Ent(\brho_{\circ,\delta}),
\end{equation}
where $c > 0$ and $C \geq 0$ are constants independent of $\delta$.
\end{lem}

\begin{proof}[Proof of Lemma~\ref{l:entropy-estimate}] We split the proof into two steps. In the first step we prove the claim for a further regularisation of $\brho_\delta$, obtained by means of a viscous approximation; we then deduce the claim for $\brho_\delta$ in the limit for the viscous approximation parameter going to zero.
\smallskip

\textit{Step 1: Viscous approximation.} Let $\varepsilon>0$; we regularise the system \eqref{f:GB-regularized-repeated} by adding the terms $\varepsilon \Delta \rho^\pm$, namely we consider
\begin{equation} \label{f:GB-regularized-repeated-epsilon}
\begin{aligned}
&\partial_t \rho^+ 
= &\div \big( \rho^+ ( \nabla V_{\delta} * (\rho^+ - \rho^-) + \nabla U ) \big) + \varepsilon \Delta \rho^+,\\
&\partial_t \rho^- 
= - &\div \big( \rho^- ( \nabla V_{\delta} * (\rho^+ - \rho^-) + \nabla U) \big) + \varepsilon \Delta \rho^-,
\end{aligned}
\end{equation}
with a smoothed initial condition given by $\brho_{\circ,\delta}^\varepsilon := \varepsilon + (1 - \varepsilon) \eta_\varepsilon * \brho_{\circ,\delta}$, where $\eta_\varepsilon$ is a mollifier defined as $\eta_\varepsilon:= \varepsilon^{-d}\eta(\frac\cdot\varepsilon)$, and $\eta\in C^\infty(\T^d)$ is nonnegative and with $\int_{\T^d}\eta=1$. Note that the components of the regularised initial datum $\brho_{\circ,\delta}^\varepsilon$ are strictly positive. 
Since $\brho\mapsto \nabla V_{\delta} * (\rho^+ - \rho^-)$ is Lipschitz continuous (as shown in \eqref{e:fieldsLip}), and $U\in C^\infty_b(\T^d)$ by assumption (U), it follows from standard parabolic theory and a bootstrap argument  that \eqref{f:GB-regularized-repeated-epsilon} has a unique classical solution $\brho_{\delta}^\varepsilon = (\rho_\delta^{\varepsilon,+}, \rho_\delta^{\varepsilon,-}) \in (C^\infty ([0,T) \times \T^d))^2$. %defsAndThms.pdf
Moreover, the components $\rho_\delta^{\varepsilon, \pm}$ are strictly positive. This can be seen by applying the comparison principle (e.g., \cite[Cor.~2.5]{Lieberman96}) to the (uncoupled) linear parabolic equations 
\begin{equation} \label{f:GB-comparison}
\partial_t \rho^\pm \mp \div \big( \rho^\pm ( \nabla V_{\delta} * (\rho_\delta^{\varepsilon,+} - \rho_\delta^{\varepsilon,-}) + \nabla U ) \big) - \varepsilon \Delta \rho^\pm = 0.
\end{equation}
Indeed, $\rho_\delta^{\varepsilon, \pm}$ is clearly a solution to \eqref{f:GB-comparison}, with initial datum $\rho_{\circ,\delta}^{\pm,\varepsilon}$; moreover, the comparison function
\begin{equation}\label{comparison_function}
  t \mapsto \Big( \inf_{\T^d} \rho_{\circ,\delta}^{\varepsilon, \pm} \Big) \exp \big( - \| \Delta V_{\delta} * (\rho^{\varepsilon,+}_\delta-\rho^{\varepsilon,-}_\delta) + \Delta U \|_{L^\infty (\T^d)} t \big)
\end{equation}
is strictly positive, and is a sub-solution of \eqref{f:GB-comparison} with initial datum below $\rho_{\circ,\delta}^{\pm,\varepsilon}$. Therefore $\rho_\delta^{\varepsilon, \pm}$ is above the comparison function \eqref{comparison_function} in $[0,T) \times \T^d$, and hence is strictly positive.

By differentiating in time the entropy of $\rho_\delta^{\varepsilon,+}$, and by using the $(\varepsilon,\delta)$-regularised equations~\eqref{f:GB-regularized-repeated-epsilon}, we have, by \eqref{def:entropy},
\begin{align*}
\frac{d}{dt} \Ent(\rho^{\varepsilon,+}_\delta) &= \frac{d}{dt}  \int_{\T^d} \rho_\delta^{\varepsilon,+}\log \rho_\delta^{\varepsilon,+} dx =  \int_{\T^d} (1+\log \rho^{\varepsilon,+}_\delta)\partial_t \rho_\delta^{\varepsilon,+} dx\\
& = \int_{\T^d} (1+\log \rho_\delta^{\varepsilon,+}) \Big(\div \big( \rho^{\varepsilon,+}_\delta ( \nabla V_{\delta} * (\rho^{\varepsilon,+}_\delta - \rho^{\varepsilon,-}_\delta) + \nabla U ) \big) + \varepsilon \Delta \rho^{\varepsilon,+}_\delta \Big) dx\\
& = - \int_{\T^d} \nabla \rho^{\varepsilon,+}_\delta\cdot \nabla (V_{\delta} * \kappa^{\varepsilon}_\delta + U) dx - \varepsilon \int_{\T^d}\frac{\big|\nabla \rho^{\varepsilon,+}_\delta\big|^2}{\rho^{\varepsilon,+}_\delta} \, dx\\
& \leq \int_{\T^d} \rho^{\varepsilon,+}_\delta \Delta (V_{\delta} * \kappa^{\varepsilon}_\delta + U) dx,
\end{align*}
where $\kappa^\varepsilon_\delta =\rho^{\varepsilon,+}_\delta-\rho^{\varepsilon,-}_\delta$. Proceeding in the same way for the entropy of $\rho^{\varepsilon,-}_\delta$, and adding up the resulting expressions, we have 
\begin{align*}
\frac{d}{dt} \Ent(\brho_\delta^\varepsilon) \leq \int_{\T^d} \kappa^\varepsilon_\delta\ \Delta \big(V_{\delta} * \kappa^\varepsilon_\delta + U \big) dx 
\leq - c \| \nabla V_\delta *\kappa^\varepsilon_\delta \|_{H^{d/2} (\mathbb T^d)}^2  +  \int_{\T^d} \Delta U \kappa^\varepsilon_\delta\, dx,
\end{align*}
where in the last step we have used Lemma \ref{l:V:props}, which provides a constant $c>0$ independent of $\delta$ and $\varepsilon$. By integrating the previous inequality over $[0,t]$, for $0 \leq t \leq T$, we finally conclude that 
\begin{equation}\label{claim-eps-0}
 \Ent(\brho^\varepsilon_\delta(t)) - \Ent(\brho^\e_{\circ, \delta}) 
 \leq - c\int_0^t \|\nabla V_{\delta}\ast \kappa^\varepsilon_\delta(\tau)\|^2_{H^{d/2}(\T^d)}\, d\tau  + \int_0^t\int_{\T^d} \Delta U \kappa^\varepsilon_\delta(\tau) dx d\tau.
\end{equation}
Since $\|\kappa^\varepsilon_\delta\|_{L^1(\T^d)}\leq 1$ and $U\in C^2_b(\T^d)$, from \eqref{claim-eps-0} we deduce that 
\begin{equation}\label{claim-eps}
 \Ent(\brho^\varepsilon_\delta(t)) + c \int_0^t \|\nabla V_{\delta}\ast \kappa^\varepsilon_\delta(\tau)\|^2_{H^{d/2}(\T^d)}\, d\tau 
 \leq C + \Ent(\brho^\e_{\circ,\delta}),
\end{equation}
namely $\brho_\delta^\varepsilon$ satisfies the claim \eqref{est:entropy}. 
\medskip

\textit{Step 2: Claim for $\brho_\delta$.} From \eqref{claim-eps} and part~\ref{l:O:bdd} of Lemma~\ref{l:Orlicz} we conclude that  
$$
\sup_{t\in[0,T]} \|\rho^{\varepsilon,\pm}_\delta(t)\|_{L\log L(\T^d)} \leq c,
$$
with a constant independent of $\delta$ and $\varepsilon$. Hence there exists some $\tilde\rho^\pm_\delta \in L^\infty(L \log L)$ such that, up to a subsequence (not relabelled) 
\begin{equation}\label{comp:1-l}
\rho^{\varepsilon,\pm}_\delta \stackrel{\ast}{\rightharpoonup} \tilde\rho^\pm_\delta \,\, \text{in } \, L^\infty(L\log L)
\end{equation}
as $\varepsilon\to 0$, i.e., in duality with $L^1( \EExp)$ (see Lemma \ref{l:Orlicz}, part \ref{l:O:Exp:dual}). 

Next we pass to the limit $\e\to 0$ in \eqref{claim-eps}. The convergence \eqref{comp:1-l} and the convexity of $\Ent$ 
allow us to pass to the limit in the first term by lower semicontinuity. Similarly, the strong convergence of 
$\brho_{\circ,\delta}^\e$ to $\brho_{\circ,\delta}$ in $L\log L(\T^d)$ gives the convergence of $\Ent(\brho^\e_{\circ,\delta})$ to $\Ent(\brho_{\circ,\delta})$ 
as $\e\to 0$ (see, e.g., \cite[Lemma 5.3]{{CannoneEl-HajjMonneauRibaud10}}. Finally, the convolution term $\nabla V_{\delta}*\kappa^\e_\delta$ is bounded and hence weakly convergent in $L^2(H^{d/2})$. By \eqref{comp:1-l},
its limit is $\nabla V_{\delta}*(\tilde\rho^+_\delta-\tilde \rho^-_\delta)$. Collecting these convergence results, we obtain from \eqref{claim-eps}, for every $t\in [0,T]$, the estimate 
\begin{equation}
\label{est:entropy_tilde}
 \Ent(\tilde{\brho}_\delta(t)) + c \int_0^t \|\nabla V_{\delta}*(\tilde\rho^+_\delta-\tilde \rho^-_\delta)(\tau)\|^2_{H^{d/2}(\T^d)}\, d\tau 
 \leq  C + \Ent(\brho_{\circ,\delta}).
\end{equation}
Finally, we prove that $\tilde\rho^\pm_\delta$ is a solution of \eqref{f:GB-regularized-repeated} with initial datum $\brho_{\circ,\delta}$. This follows by letting $\varepsilon\to 0$ in the weak form of the $(\varepsilon,\delta)$-regularised equations~\eqref{f:GB-regularized-repeated-epsilon} satisfied by $\brho_\delta^\e$, proceeding similarly as in Step 1 of the proof of Lemma \ref{l:dist-bmu-brho_n}.  By the uniqueness result in Lemma \ref{l:dist-bmu-brho_n} we conclude that $\tilde\rho^\pm_\delta = \rho^\pm_\delta$, and hence \eqref{est:entropy_tilde} gives the claim and concludes the proof.
\end{proof}

We are now ready to prove Theorem~\ref{thm:delta_to_0}, namely to connect the regularised solutions $(\brho_\delta)$ of~\eqref{f:GB-regularized-repeated}, with initial approximate datum $\brho_{\circ,\delta}$, with a solution of the unregularised Groma-Balogh equations with initial datum $\brho_\circ$.

\begin{proof}[Proof of Theorem~\ref{thm:delta_to_0}] Note that the existence of $\brho_{\circ,\delta} \in \big( L \log L (\mathbb T^d) \big)^2 \cap \mathcal P (\mathbb T^d \times \{\pm1\})$ such that $\brho_{\circ,\delta}\xweakto \ast \brho_\circ$ as $\delta\to 0$ in measure, and $\Ent(\brho_{\circ,\delta})\leq c$ uniformly in $\delta$, follows from Lemma \ref{lemm:approx}.
We split the proof into three steps.
\smallskip

\textit{Step 1: Compactness of $\brho_\delta$.} Since $\Ent(\brho_{\circ,\delta})\leq c$ uniformly in $\delta$, from  Lemma~\ref{l:entropy-estimate} and part~\ref{l:O:bdd} of Lemma~\ref{l:Orlicz}, we 
have that $\brho_\delta=(\rho^+_\delta, \rho^-_\delta)$ and $\kappa_\delta = \rho^+_\delta-\rho^-_\delta$ satisfy the following \emph{a priori} estimates:
\begin{equation}\label{apriori_rhon}
\sup_{t\in[0,T]} \|\rho^\pm_\delta(t)\|_{L\log L(\T^d)} \leq c \quad \text{and} \quad \int_0^T \|\nabla V_{\delta}*\kappa_\delta (t)\|_{H^{d/2}(\T^d)}^2 dt \leq c.
\end{equation}
By the first bound in \eqref{apriori_rhon} we can extract a subsequence (not relabelled) such that 
\begin{equation}\label{comp:1}
\rho^\pm_\delta \stackrel{\ast}{\rightharpoonup} \rho^\pm \,\, \text{in } \, L^\infty(L\log L),
\end{equation}
as $\delta\to 0$, i.e., in duality with $L^1( \EExp)$, for some $\rho^\pm \in L^\infty(L \log L)$.

Next we show that $\brho(t) \in \P(\T^d\times \{\pm1\})$ for almost all $t$. Let $\varphi \in C([0,T])$ be a test function (note that $\varphi\in L^1( \EExp)$, since $\EExp(\T^d)$ contains the constants). We compute
\begin{multline*}
\int_0^T \varphi(t) dt = \int_0^T\int_{\T^d} \Bigl(\varphi(t)\rho_\delta^+(t,x) + \varphi(t)\rho_\delta^-(t,x)\Bigr)\,dxdt\\
 \stackrel{\delta\to 0}{\longrightarrow} \int_0^T\int_{\T^d} \Bigl(\varphi(t)\rho^+(t,x) + \varphi(t)\rho^-(t,x)\Bigr)\,dxdt,
\end{multline*}
and thus $\int_{\T^d}(\rho^+(t,x)+\rho^-(t,x)) dx = 1$ for almost all $t$.
\smallskip

\textit{Step 2: Compactness of $\nabla V_\delta\ast\kappa_\delta$ in $L^2(\EExp)$.} 
From the second estimate in \eqref{apriori_rhon} we deduce that $\nabla V_{\delta}*\kappa_\delta$ is bounded in $L^2(H^{d/2})$, and hence, up to a subsequence, weakly convergent to a limit $f\in L^2(H^{d/2})$. We now show that the convergence is actually strong in $L^2(\EExp)$, and that $f = \nabla V\ast \kappa$, where $\kappa = \rho^+-\rho^-$.

We first note that $\rho_\delta^\pm\nabla V_{\delta}*\kappa_\delta$ is bounded in $L^2(L\log^{1/2}L)$. Indeed, by Lemma \ref{l:Orlicz}\,\ref{l:O:Hol:LlogL} and \ref{l:O:H1:Exp:ct} (integrated in time), and by \eqref{apriori_rhon},
\begin{align}\label{estimateM}
\bigg(\int_0^T\|\rho_\delta^\pm (t) \nabla V_{\delta} * \kappa_\delta(t) \|^2_{L\log^{1/2}L(\T^d)} dt \bigg)^{1/2}
&\leq C \|\rho_\delta^\pm\|_{L^\infty(L\log L)} \|\nabla V_{\delta} * \kappa_\delta\|_{L^2(H^{d/2})} 
\leq C.
\end{align}

Using \eqref{estimateM} and \eqref{apriori_rhon}, we show that $\partial_t \rho_\delta^\pm $ is bounded in $L^2(H^{-(d+1)})$. Indeed, by Lemma \ref{l:Orlicz}\,\ref{l:O:H1:Exp:ct}, and by using \eqref{f:GB-regularized-repeated}, we have that for any $\varphi\in C_c^\infty((0,T)\times \T^d)$
\begin{align}\label{stima:duale}
&\bigg| \int_0^T \int_{\T^d} \rho_\delta^\pm(t,x) \partial_t \varphi (t,x) dx dt \bigg|
\leq \int_0^T \int_{\T^d} \Big| \rho_\delta^\pm(t,x) \, \nabla \varphi(t,x) \cdot  \nabla (V_\delta*\kappa_\delta + U)(t,x) \Big| dx dt\nonumber\\
&\qquad \leq c\big(\|\nabla \varphi\|_{L^2(\Exp_2)}\|\rho^\pm_\delta\nabla V_{\delta}* \kappa_\delta\|_{L^2(L\log^{1/2}L)} + \|\nabla \varphi\|_{L^2(\Exp)} \|\rho^\pm_\delta\|_{L^\infty(L\log L)}\Big) \nonumber\\
&\qquad \leq c\|\varphi\|_{L^2(H^{d/2+1})}\leq c\|\varphi\|_{L^2(H^{d+1})}.
\end{align}

Using $\partial_t \kappa_\delta \in L^2(H^{-(d+1)})$ we show next that $\partial_t (\nabla V_{\delta} * \kappa_\delta)$ is bounded in $L^2(H^{-(d+1)})$. First, 
we note that by taking the difference of the equations \eqref{f:GB-regularized-repeated} we have that $\partial_i V_\delta\ast \kappa_\delta$ satisfies, for $i=1,\dots,d$,
\begin{align*}
\partial_t (\partial_i V_\delta\ast\kappa_\delta) = \div \big( \partial_i V_\delta\ast(\rho^+_\delta +\rho^-_\delta) ( \nabla V_{\delta} *\kappa_\delta + \nabla U ) \big)
\end{align*}
in the sense of distributions, namely for any $\varphi\in C_c^\infty((0,T)\times \T^d)$
\begin{align}\label{distrM}
\int_0^T \int_{\T^d} (\partial_i V_\delta\ast\kappa_\delta) \partial_t \varphi\, dx dt  = 
- \int_0^T \int_{\T^d} (\rho^+_\delta +\rho^-_\delta) ( \nabla V_{\delta} *\kappa_\delta + \nabla U )\cdot (\partial_i V_\delta\ast \nabla \varphi) dx dt.  
\end{align}
Using \eqref{estimateM}, we can then estimate the right-hand side of \eqref{distrM} in absolute value by
\begin{multline*}
c \|(\rho_\delta^+ + \rho_\delta^-) \big(\nabla V_{\delta} * \kappa_\delta + \nabla U\big) \|^2_{L^2(L\log^{1/2}L)} \|\partial_i V_\delta\ast \nabla \varphi\|_{L^2(\Exp_2)} \\
\leq c \|\partial_i V_\delta\ast \nabla\varphi\|_{L^2(H^{d/2})} 
\leq c \|\partial_i V_\delta\ast \varphi\|_{L^2(H^{d+1})},
\end{multline*}
which is uniformly bounded in $\delta$ by $(V2)$, since 
$$
 \|\nabla V_\delta \ast \varphi\|^2_{L^2(H^{d+1})} = \int_0^T \sum_{k \in \Z^d} (1 + |k|^2)^{d+1} \big| \big[ \widehat{\nabla V_{\delta}} \big]_k \big|^2 |\hat\varphi_k|^2 dt
\leq c  \|\varphi\|^2_{L^2(H^{d+1})}.
$$

By applying Lemma~\ref{l:ALS} with $p=q=2$ to the triple $H^{d/2} (\T^d) \subset\subset \EExp (\T^d) \subset H^{-(d+1)} (\T^d)$ (see  Lemma\,\ref{l:Orlicz}\,\ref{l:O:H1:Exp:cp}), we find that 
$\nabla V_{\delta} * \kappa_\delta$ is compact in $L^2(\EExp)$, namely it converges strongly to $f$.

\smallskip

\noindent
Finally, we show that $f = \nabla V\ast \kappa$. To see this, we compute the distributional limit of the sequence $\nabla V_{\delta} * \kappa_\delta$. Let $\psi \in C_c^\infty((0,T)\times \T^d; \R^d)$ be a test function; then we have 
\begin{align*}
\int_0^T\int_{\T^d} (\nabla V_{\delta} * \kappa_\delta)(t,x) \cdot \psi (t,x) dxdt &= - \int_0^T\int_{\T^d} (V_{\delta} * \kappa_\delta)(t,x) \div \psi (t,x) dxdt \\
&= - \int_0^T\int_{\T^d} (V_{\delta} * \div \psi)(t,x) \kappa_\delta (t,x) dxdt.
\end{align*}
By using \eqref{comp:1} and assumption $(V3)$, we have for the last term that 
$$
\lim_{\delta\to 0} \int_0^T\int_{\T^d} (V_{\delta} * \div \psi)(t,x) \kappa_\delta (t,x) dxdt = \int_0^T\int_{\T^d} (V* \div \psi)(t,x) \kappa(t,x) dxdt,
$$
and by the uniqueness of the limit of $\nabla V_{\delta} * \kappa_\delta$ we deduce that $f = \nabla V\ast \kappa$.

\medskip

\textit{Step 3: The limit $\rho^\pm$ is a solution of the unregularised Groma-Balogh equations.}
The solution $\brho_\delta\in (L^\infty(L \log L))^2$ of the regularised equation~\eqref{f:GB-regularized-repeated} with initial datum $\brho_{\circ,\delta}$ satisfies the weak equations
\begin{align}\label{weakrhon}
&\int_0^T\int_{\T^d} \rho^\pm_\delta (t,x) \Big( \partial_t \varphi(t,x) \mp \nabla \varphi(t,x) \cdot  \nabla (V_{\delta}*\kappa_\delta)(t,x) \mp \nabla \varphi(t,x) \cdot \nabla U(x) \Bigr) dx dt \\
&+ \int_{\T^d} \rho_{\circ,\delta}^\pm(x) \varphi(0,x) dx =0 \nonumber
\end{align}
for every $\varphi \in C_c^\infty ([0,T)\times \T^d)$. We now show that, thanks to the compactness results in Steps 1 and 2, we can pass to the limit in~\eqref{weakrhon} as $\delta\to 0$, hence showing that 
$\brho$ is a solution of the unregularised Groma-Balogh equations, in the sense of Definition~\ref{def:GB-sol-isotropic}, with initial datum $\brho_\circ$.
\smallskip

By \eqref{comp:1}, we can pass to the limit in the first and in the third terms of the first integral in equations~\eqref{weakrhon}. For the last term, the convergence is immediate, since 
$\brho_{\circ,\delta}\xweakto \ast \brho_\circ$ as $\delta\to 0$.
To pass to the limit in the nonlocal terms of equations~\eqref{weakrhon}, note that, by \eqref{comp:1} and by Step 2, 
\begin{align*}
&\int_0^T\int_{\T^d} \big(\rho^\pm_\delta \nabla V_{\delta}*\kappa_\delta - \rho^\pm \nabla V*\kappa\big)\cdot \nabla \varphi  \,dx dt\\
&=\int_0^T\int_{\T^d} \rho^\pm_\delta \big(\nabla V_{\delta}*\kappa_\delta - \nabla V*\kappa\big)\cdot \nabla \varphi  \,dx dt 
+ \int_0^T\int_{\T^d} (\rho^\pm_\delta - \rho^\pm) (\nabla V*\kappa)\cdot \nabla \varphi  \,dx dt \stackrel{\delta\to 0}{\longrightarrow} 0.
\end{align*}
This concludes the proof of the theorem.
\end{proof}

Finally, we prove our main result, Theorem~\ref{t:evoConv}, by combining Lemma~\ref{l:dist-bmu-brho_n} and Theorem~\ref{thm:delta_to_0}.

\begin{proof}[Proof of Theorem~\ref{t:evoConv}] Let $(\delta_n)$ be an infinitesimal sequence satisfying assumption \eqref{cond:evoConv-conv-initial}. Let $t\mapsto \bmu_n(t)$ be the empirical measure of the solution $t\mapsto x^n(t)$ of~\eqref{for:evo:dyncs:edge:2D-repeated} with initial datum $\bmu_{\circ,n}$, and let $\brho_{\delta_n} \in C([0,T]; \mathcal P (\mathbb T^d \times \{\pm1\}))\cap ( L^\infty(0,T; (L \log L(\T^d))^2)$ be the solution of the regularised equation~\eqref{f:GB-regularized-repeated} with initial datum $\brho_{\circ, n}$, whose existence and uniqueness are established in Lemma \ref{l:dist-bmu-brho_n}, and whose bounds are proved in Lemma \ref{l:entropy-estimate}.
 
First, by Lemma~\ref{l:dist-bmu-brho_n} and assumption \eqref{cond:evoConv-conv-initial}, we can extract a subsequence of $\delta_n$ (not relabelled) along which we have 
\begin{align} \label{p:dWT:conv}
\sup_{t\in [0,T]}{\bW} (\bmu_{n}(t), \brho_{\delta_n}(t))\xto{n \to \infty} 0.
\end{align}
Moreover, we claim that
\begin{subequations}
  \begin{align} \label{p:bdd:seq1}
(\brho_{\delta_n})
&\text{ is bounded in } L^\infty (0,T; (L\log L(\T^d))^2), \text{ and} \\ \label{p:bdd:seq2}
(\partial_t \brho_{\delta_n})
&\text{ is bounded in } L^2 (0,T; (W^{-1,1} (\T^d))^2).
\end{align}  
\end{subequations}
Property \eqref{p:bdd:seq1} is given directly by \eqref{apriori_rhon}; \eqref{p:bdd:seq2} follows from \eqref{estimateM} since, for any $\varphi\in C_c^\infty((0,T)\times \T^d)$, proceeding similarly as in \eqref{stima:duale},
\begin{align*}
\bigg| \int_0^T \int_{\T^d} \rho_{\delta_n}^\pm(t,x) \partial_t \varphi (t,x) dx dt \bigg| 
\leq c \, \|\nabla \varphi\|_{L^2(\Exp_2)} \leq c \, \|\varphi\|_{L^2(W^{1,\infty})},
\end{align*}
where we have used the continuous embedding $L^\infty(\T^d)\hookrightarrow \Exp_2(\T^d)$ (see, e.g., \cite[Lemma 1]{Kuhn03}). 

Thanks to \eqref{p:bdd:seq1}-\eqref{p:bdd:seq2}, we apply Lemma \ref{l:ALS} to $X_0 = (L\log L(\T^d))^2$ and $X=X_1 = (W^{-1,1} (\T^d))^2$ to deduce compactness of $(\brho_{\delta_n})$ in $C^0([0,T]; (W^{-1,1} (\T^d))^2)$. In particular, up to a subsequence, 
\begin{equation}\label{BLn-convergence}
\sup_{t\in [0,T]} \|\brho_{\delta_n}(t) -  \brho (t)\|_{\textrm{BL}}^\ast \to 0,
\end{equation}
where, by Theorem~\ref{thm:delta_to_0}, $\brho$ is a solution of the unregularised Groma-Balogh equations with initial datum $\brho_\circ$, in the sense of Definition~\ref{def:GB-sol-isotropic}. Since the dual bounded Lipschitz norm, when restricted to $\P (\T^d \times \{\pm1\} )$, is equivalent to $\bW$ (by Kantorovich duality), we deduce from \eqref{BLn-convergence} that 
\begin{align} \label{p:dWT:conv-2}
\sup_{t\in [0,T]}{\bW} (\brho_{\delta_n}(t), \brho (t))\xto{n \to \infty} 0.
\end{align}
In conclusion, \eqref{p:dWT:conv} and \eqref{p:dWT:conv-2} imply that $\bW(\bmu_n(t),\brho(t))\to 0$ along a subsequence as $n \to \infty$, uniformly in $t \in [0,T]$. This concludes the proof of Theorem \ref{t:evoConv}.
\end{proof}

\subsection{Case 2: Two-dimensional domain, slip plane-confined evolution.} Here we consider the anisotropic case in two dimensions, as in Section~\ref{subsec:description-confined}. 
The main result of this section is Theorem \ref{t:evoConv:conf}, which is the analogue of Theorem \ref{t:evoConv} in this framework. 

\medskip

In analogy with Definition \ref{def:GB-sol-isotropic}, we define the concept of solution for the unregularised and anisotropic evolution \eqref{f:GB-slip-plane-confined} as follows.

\begin{defn}
\label{def:GB-sol-isotropic:conf}
Let $\brho_\circ =(\rho^+_\circ, \rho^-_\circ)\in \big( L \log L (\mathbb T^2) \big)^2 \cap\mathcal P(\T^2\times \{\pm1\})$, let $T>0$, and let $\brho=(\rho^+, \rho^-): [0,T]\to \mathcal P(\mathbb T^2\times \{\pm1\})$. We say that $\brho$ is a solution of~\eqref{f:GB-slip-plane-confined} in $[0,T]$, with initial datum $\brho_\circ$, if 
\begin{itemize}
\smallskip
\item[$(1)$] $\rho^+, \rho^-\in L^\infty(0,T;L\log L(\T^2))$ and $\partial_1 V*\kappa\in L^1(0,T;H^1(\T^2))$, where $\kappa = \rho^+-\rho^-$;
\item[$(2)$] For every $\varphi, \psi \in C_c^\infty ([0,T)\times \T^2)$
\begin{subequations}
\begin{align*}
&\int_0^T\int_{\T^2} \rho^+ (t,x) \big(\partial_t \varphi(t,x) - \partial_1 \varphi(t,x) \, \partial_1 (V*\kappa + U)(t,x) \big) dx dt + \int_{\T^2} \rho_\circ^+(x) \varphi(0,x) dx =0,\\
&\int_0^T\int_{\T^2} \rho^- (t,x) \big(\partial_t \psi(t,x) + \partial_1 \psi(t,x) \,  \partial_1 (V*\kappa + U)(t,x) \big) dx dt + \int_{\T^2} \rho_\circ^-(x) \psi(0,x) dx =0.
\end{align*}
\end{subequations}
\end{itemize}
\end{defn}

\medskip

\noindent
We set, as in the previous section, $\lambda_{\delta} := \|D^2 V_{\delta}\|_{L^\infty(\T^2)}$. The main theorem of this section is the following. 

\begin{thm}[Evolutionary convergence] \label{t:evoConv:conf} 
Let $\brho_\circ \in \big( L \log L (\mathbb T^2) \big)^2 \cap \mathcal P (\mathbb T^2 \times \{\pm1\})$, and let $(\brho_{\circ,n}), (\bmu_{\circ,n})\subset \mathcal P (\mathbb T^2 \times \{\pm1\})$ be approximating sequences for $\brho_\circ$ as in Lemma \ref{lemm:approx}. Let $T>0$ be fixed, and let $\delta_n>0$ be such that 
\begin{equation}
\label{cond:evoConv-conv-initial-ani}
\exp (3 \lambda_{\delta_n} T) \, \bW(\bmu_{\circ,n}, \brho_{\circ,n}) \to 0, \qquad \text{as }n \to \infty.
\end{equation}
Let $t\mapsto \bmu_n(t)$ be the empirical measure of the solution $t\mapsto x^n(t)$ of~\eqref{for:evo:dyncs:edge:2D-confined} on $[0,T]$, with $\delta=\delta_n$ and initial datum $\bmu_{\circ,n}$.

Then there exists the limit $\brho(t):=\lim_{n\to \infty}\bmu_n(t)$ with respect to $\bW$, uniformly in $t \in [0,T]$, where $\brho$ is a solution of the unregularised Groma-Balogh equations~\eqref{f:GB-slip-plane-confined} in the sense of Definition~\ref{def:GB-sol-isotropic:conf} with initial datum $\brho_\circ$.
\end{thm}

\noindent
\textit{Structure of the proof of Theorem \ref{t:evoConv:conf}.}
Again, we prove this theorem by constructing an intermediate solution $\brho_{\delta_n}$ of the regularised equations \eqref{f:GB-regularized-confined} with the modified initial condition $\brho_{\circ,n}$ with $n\rho^\pm_{\circ,n}(\T^2)\in \N$, and show that $\brho_{\delta_n}$ is close to $\bmu_n$ with respect to $\bW$, uniformly in time, and converges to a solution $\brho$ of~\eqref{f:GB-slip-plane-confined}.

\begin{proof}
We split the proof into three steps.

\smallskip

\noindent
\textit{Step 1: Existence, uniqueness and Gronwall estimate for \eqref{f:GB-regularized-confined} for fixed $\delta$.} The proof of the existence of a solution for \eqref{f:GB-regularized-confined} follows by discretisation via empirical measures, exactly as in Step 1 of the proof of Lemma \ref{l:dist-bmu-brho_n}.  Similarly, the derivation of the Gronwall estimate can be done as in Step 2 of the proof of Lemma \ref{l:dist-bmu-brho_n}.

\smallskip

\noindent
\textit{Step 2: Convergence of $\brho_{\delta}$ to a solution $\brho$ of~\eqref{f:GB-slip-plane-confined}.} We split this step into two sub-steps: The derivation of a-priori bounds for $\brho_\delta$, in the spirit of Lemma~\ref{l:entropy-estimate}, and the limit process as $\delta\to 0$.

\noindent
\textit{Step 2.1: Bounds for $\brho_\delta$ for fixed $\delta$.} Let $\varepsilon>0$ and let  $\brho_{\delta}^\varepsilon = (\rho_\delta^{\varepsilon,+}, \rho_\delta^{\varepsilon,-}) \in C^\infty ([0,T) \times \T^2)$ be the solution of the regularisation of the system \eqref{f:GB-regularized-confined} obtained by adding the term $\varepsilon \Delta \rho^\pm$ to both equations as for \eqref{f:GB-regularized-repeated-epsilon}, and by regularising the initial datum $\brho_{\circ,\delta}$. Following the proof of Lemma~\ref{l:entropy-estimate} and 
using Remark \ref{rem:11}, it is easy to show that for every $t\in [0,T]$,
\begin{equation}\label{claim-eps-11}
 \Ent(\brho^\varepsilon_\delta(t)) + \int_0^t \|\partial_1 V_{\delta}\ast \kappa^\varepsilon_\delta(\tau)\|^2_{H^{1}(\T^2)}\, d\tau \leq c + \Ent(\brho^\e_{\circ,\delta}),
\end{equation}
that the limit of $\brho^\varepsilon_\delta$ as $\e\to 0$ is the solution $\brho_\delta$ of~\eqref{f:GB-regularized-repeated} with initial datum $\brho_{\circ,\delta}$, and that $\brho_\delta$ satisfies 
\begin{equation*}
 \Ent(\brho_\delta(t)) + \int_0^t \|\partial_1 V_{\delta}\ast \kappa_\delta(\tau)\|^2_{H^{1}(\T^2)}\, d\tau \leq c + \Ent(\brho_{\circ,\delta}),
\end{equation*}
by letting $\e\to 0$ in \eqref{claim-eps-11}.

\smallskip

\noindent
\textit{Step 2.2: Convergence of $\brho_\delta$ to a solution of~\eqref{f:GB-slip-plane-confined}.} Here we can follow exactly the proof of Theorem~\ref{thm:delta_to_0} to deduce convergence of $\brho_\delta$ to a solution $\brho$ of~\eqref{f:GB-slip-plane-confined}, with respect to the $\textrm{weak}^\ast$ convergence in $L^\infty(0,T; (L\log L (\T^2))^2 )$. 

\smallskip

\noindent
\textit{Step 3: Conclusion.} Let $(\delta_n)$ be an infinitesimal sequence satisfying assumption \eqref{cond:evoConv-conv-initial-ani}. Let $t\mapsto \bmu_n(t)$ be the empirical measure of the solution $t\mapsto x^n(t)$ of~\eqref{for:evo:dyncs:edge:2D-confined} with initial datum $\bmu_{\circ,n}$, and let $\brho_{\delta_n} \in C([0,T]; \mathcal P (\mathbb T^2 \times \{\pm1\}))\cap ( L^\infty(0,T; (L \log L(\T^2))^2)$ be the solution of the regularised equation~\eqref{f:GB-regularized-repeated} with initial datum $\brho_{\circ,n}$, whose existence and uniqueness are established in Step 1 of the proof, and whose bounds are proved in Step 2.
 
Then, proceeding as in the proof of Theorem~\ref{t:evoConv}, we can extract a time-independent subsequence of $n$ (not relabelled) along which we have simultaneously
\begin{align*} 
\bW (\bmu_n(t), \brho_{\delta_n}(t)) \xto{n \to \infty} 0, 
\quad \text{and} \quad 
\bW(\brho_{\delta_n}(t), \brho(t)) \xto{n \to \infty} 0,
 \end{align*}
 uniformly in time,
 where  $\brho$ is a solution of the unregularised Groma-Balogh equations \eqref{f:GB-slip-plane-confined}, in the sense of Definition~\ref{def:GB-sol-isotropic:conf}, with initial datum $\brho_\circ$.
\end{proof}

%%%%%%%%%%%%%%%%%%%%%%%%%%%%%%%%%%%%%%%%%%%%%%%%%%%%%
%%%%%%%%%%%%%%%%%%%%%%%%%%%%%%%%%%%%%%%%%%%%%%%%%%%%%
%%%%%%%%%%%%%%%%%%%%%%%%%%%%%%%%%%%%%%%%%%%%%%%%%%%%%
%

\section{Non-convergence: A uniform distribution of short dipoles} \label{section:dipole}

In this section we construct counterexamples to the discrete-to-continuum convergence of gradient flows for the isotropic and the slip-confined cases of Section~\ref{sec:setups}. More precisely, for the isotropic case in dimension $d=1,2$ and for the slip-confined case in dimension $d=2$, we construct a family of initial data for the particle systems \eqref{for:evo:dyncs:edge:2D} such that the resulting discrete solutions do converge to a continuum limit, but this limit does not solve the corresponding Groma-Balogh equations \eqref{f:GBintro}. In Section~\ref{ss:cvb:sketch} we explain the idea behind this construction, and heuristically describe how this choice leads to non-convergence. In Section~ \ref{ss:cvb:1D} we state and prove the non-convergence result for the isotropic one-dimensional (Theorem \ref{t:CvB:1D}) and two-dimensional cases (Theorem \ref{t:CvB:2D}). Section \ref{ss:cvb:C2} deals with the anisotropic case of edge dislocations.

\subsection{Choice of the initial conditions $(x_\circ^n,b_\circ^n)$.}
\label{ss:cvb:sketch}
The strategy of the counterexamples is as follows. We assume that the regularisation parameter $\delta_n$ converges to zero sufficiently quickly, and we construct a sequence of initial data for which the resulting solutions of \eqref{for:evo:dyncs:edge:2D} are \emph{constant in time}, at least approximately. The continuum limit of these solutions is then constant in time, and therefore is not a solution of the Groma-Balogh equations \eqref{f:GBintro}, unless $\nabla U  \equiv 0$.

In Figure \ref{fig:dipole} we sketch our initial condition $(x_\circ^n,b_\circ^n)$ in the case of edge dislocations in a two-dimensional domain. 
We choose to position $\frac n2$  \emph{dipoles} of width of order $\delta_n$ at a distance of order at least $\tfrac{1}{\sqrt n}$ from one another, under the assumption that $\delta_n\ll \frac1{n}$ (the precise assumption on the relation between $\delta_n$ and $n$ will be specified in each section).

\begin{figure}[h!]
\centering
\begin{tikzpicture}[scale=0.11]
    \def \r {0.15}
      
    \foreach \point in {(5,3), (20,9), (8,22), (18, 36), (32, 27), (34, 5)}{
    \begin{scope}[shift=\point,scale=1, rotate=270]
        \fill (1 - \r, -1) rectangle (1 + \r, 1); 
        \fill (-1, -\r) rectangle (1, \r); 
        \fill (-1, 0) circle (\r);
        \fill (1, 1) circle (\r);
        \fill (1, -1) circle (\r);  
    \end{scope}
    }
    
    \foreach \point in {(3,3), (18,9), (7,19), (20, 36), (33, 24), (36, 5)}{
    \begin{scope}[shift=\point,scale=1, rotate=270]
        \fill (-1 - \r, -1) rectangle (-1 + \r, 1); 
        \fill (-1, -\r) rectangle (1, \r); 
        \fill (1, 0) circle (\r);
        \fill (-1, 1) circle (\r);
        \fill (-1, -1) circle (\r);  
    \end{scope}
    }
    
    \foreach \point in {(4,3), (19,9), (7.5,20.5), (19, 36), (32.5, 25.5), (35, 5)}{
    \begin{scope}[shift=\point]
        \draw (0,0) circle (3);
    \end{scope}
    }
    
    \draw[<->] (4 + 0.59, 3 + 2.94) -- (7.5 - 0.59, 20.5 - 2.94) node[midway, right]{$\geq \frac{c}{\sqrt n}$};
    \draw[<->] (32, 9) -- (38, 9) node[midway, above]{$\sim \delta_n$};
    
    \draw[dashed] (32.5, 28.5) -- (80,40);
    \draw[dashed] (32.5, 22.5) -- (80,0);
    \begin{scope}[shift = {(80,20)}]
        \fill[white] (0,0) circle (20);
        \draw (0,0) circle (20);
        \fill (0,0) circle (4*\r);
        \draw (0,-\r) node[below]{$z_{\circ,i}^n$};
        
        \begin{scope}[rotate=180]
        \draw (0,0) -- (8,-8) node[midway, right]{$d_{\circ,i}^n$};
        \draw (6,-7) -- (8,-8) -- (7,-6);
        \draw[dotted] (0,0) -- (-8,8);
        \end{scope}
        
        \draw (8,-8) node[anchor = south west]{$x_{\circ,i}^{n,-}$};
        \begin{scope}[shift={(8,-10)}, scale=2, rotate=270]
          \fill (-1 - \r, -1) rectangle (-1 + \r, 1); 
          \fill (-1, -\r) rectangle (1, \r); 
          \fill (1, 0) circle (\r);
          \fill (-1, 1) circle (\r);
          \fill (-1, -1) circle (\r);  
        \end{scope}
        
        \draw (-8,8) node[anchor = south west]{$x_{\circ,i}^{n,+}$};
        \begin{scope}[shift={(-8,10)},scale=2, rotate=270]
        \fill (1 - \r, -1) rectangle (1 + \r, 1); 
        \fill (-1, -\r) rectangle (1, \r); 
        \fill (-1, 0) circle (\r);
        \fill (1, 1) circle (\r);
        \fill (1, -1) circle (\r);   
    \end{scope}
    \end{scope}
\end{tikzpicture} 
\caption{A configuration of $n = 12$ edge dislocations arranged in $\tfrac n2 = 6$ dipoles. The dipole $(x_{\circ,i}^{n,+}, x_{\circ,i}^{n,-})$, with the positive dislocation at $x_{\circ,i}^{n,+}$ and the negative dislocation at $x_{\circ,i}^{n,-}$, is highlighted, as well as its mid-point $z_{\circ,i}^n$ and the dipole half-width vector $d_{\circ,i}^n$ (see \eqref{variables_zd} for the precise definition).}
\label{fig:dipole}
\end{figure}
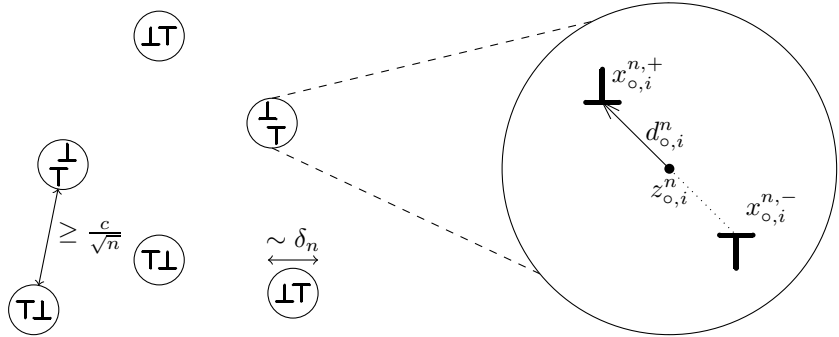

The idea behind our choice for the initial condition is the following. Since $V$ blows up logarithmically at $0$, $-V_{\delta_n}$ has a deep well around $0$ of width proportional to $\delta_n$. We call a pair of a positive and a negative dislocation at $(x_{\circ,i}^{n,+}, x_{\circ,i}^{n,-})$ a \emph{short} dipole if it `fits' in this well (namely, if its width is smaller than the width of the well). By the logarithmic nature of the singularity of $V_{\delta_n}$, the walls of this well are $\frac1{\delta_n}$ steep, so that it takes a large force to break this dipole apart. 

On the other hand, the dipole-dipole interaction is weak: indeed the force field generated by a dipole at $z_{\circ,i}^n$ is small at any $y \in \T^2$ with $|y - z_{\circ,i}^n| > \tfrac{c}{\sqrt n}$, since the interaction forces $\nabla V_{\delta_n}(x_{\circ,i}^{n,+} - y)$ and $-\nabla V_{\delta_n}(x_{\circ,i}^{n,-} - y)$ generated by the positive and the negative dislocations in the dipole cancel out up to an error  of order $\delta_n$. 
Similarly, also the net effect of the force $-\nabla U$ generated by the external potential $U$ on the dipole is of order $\delta_n$.

Hence, the evolution of the short dipoles governed by \eqref{for:evo:dyncs:edge:2D} is slow in time, and each dipole does not break apart. By choosing the initial positions carefully, and by making appropriate assumptions on the rate $\delta_n\to0$, we can then prove that the solution $x^n(t)$ of \eqref{for:evo:dyncs:edge:2D} is approximately stationary, namely $x^n(t) \approx x^n_\circ$ for all $t \in [0,T]$, up to an error term that is infinitesimal as $n \to \infty$. 

Now, let $\brho_\circ = (\rho_\circ^+,\rho_\circ^-) \in \mathcal{P}(\T^2\times\{\pm1\})$ be such that $\rho_\circ^+ = \rho_\circ^- \in L^\infty (\T^2)$. The assumption that $\rho_\circ^+ = \rho_\circ^-$ is essential since it allows us to approximate $\brho_\circ$ with short dipoles $(x_\circ^n,b_\circ^n)$ as in Figure \ref{fig:dipole}, with $c = 1/\| \rho_\circ^+ \|_\infty$. Then, the fact that the discrete solutions $x^n(t)$ of \eqref{for:evo:dyncs:edge:2D} with initial data $(x_\circ^n,b_\circ^n)$ are approximately constant implies that $\bmu_n (t) \weakto \brho_\circ$ as $n \to \infty$ for all $t \in [0, T]$, where $\bmu_n(t)$ is the empirical measure associated to $(x^n(t),b_\circ^n)$.  However, if we assume that $\nabla U \not \equiv 0$ on $\supp \brho_\circ$, then the limit stationary solution $\brho_\circ$ is not a solution to the limit equation \eqref{f:GBintro}. 
\smallskip

In the remainder of this section we make the arguments above rigorous in the isotropic case in dimension one and two (Section~\ref{ss:cvb:1D}), and in the two-dimensional slip-confined case (Section~\ref{ss:cvb:C2}).

\subsection{The isotropic case: $d=1$ and $d=2$.}
\label{ss:cvb:1D}

Let $\T^d$ be the $d$-dimensional flat torus, and let $\|\cdot\|_{\Td}$ denote the distance on $\T^d$ defined in \eqref{norm:T}. 
Let $\delta_n>0$ for every $n\in \N$. We assume that $V, V_{\delta_n}, U: \T^d \to \R$ satisfy the following assumptions, in analogy with the ones listed in Section~\ref{sec:setups} (here the subscript I stands for `isotropic'):
\begin{itemize}
\item[$(V1)_{\textrm{I}}$]  $\big\| \|s\|_{\Td}^2 D^2V(s) \big\|_{L^\infty(\T^d)} < \infty$, and $V(x) = V(-x)$;
\item[$(V2)_{\textrm{I}}$] $V_{\delta_n} \in C^2(\T^d)$, and $V_{\delta_n}(x) = V_{\delta_n}(-x)$;
\vspace{.1cm}
\item[$(V3)_{\textrm{I}}$] $\left\{
\begin{tabular}{@{}l@{}}  \smallskip
    ($d=1$) $\exists \, c, \gamma > 0$ such that $V'_{\delta_n}(2 \gamma \delta_n) \leq - c / \delta_n$ for $n$ large; \\
    \smallskip
    ($d=2$) there exist $C, c > 0$, and a closed and simple $C^1$-curve $\Gamma_n =\partial B_n$,\\ $B_n\subset B(0, C \delta_n)$, such that 
    $\max_{\Gamma_n} \big(\nabla V_{\delta_n}\cdot \nu_n\big) \leq - c / \delta_n$ for $n$ large, where\\ $\nu_n$ is the outward pointing normal to $\Gamma_n$;
\end{tabular}
\right.$
\vspace{.1cm}
\item[$(V4)_{\textrm{I}}$] 
$\displaystyle \delta_n  \| \nabla V_{\delta_n} \|_{L^\infty(\T^d)}$ is bounded in $n$; 
\item[$(V5)_{\textrm{I}}$]  $\forall\ c > 0$  $\exists\ C > 0$ such that $\sup\Big\{\| D^2V_{\delta_n} (s)\|: \|s\|_{\Td} > \frac{c}{n^{1/d}} \Big\}  \leq C n^{\frac2d}$ for $n$ large;
\item[$(U)_{\textrm{I}}$]  $U \in C^2(\T^d)$.
\end{itemize}

The assumptions above are quite natural and are satisfied by the one-dimensional model potential $V_{\delta_n}(s)= -\frac12 \log(s^2+\delta_n^2)$. Assumption $(V5)_{\textrm{I}}$ mimics the bound $(V1)_{\textrm{I}}$ on~$V$, and forces the regularisation of the singularity to be localised in an interval around $0$ with length asymptotically smaller than $\tfrac{1}{n^{1/d}}$ (which is the typical dipole-dipole distance in dimension $d$). Assumption $(V4)_{\textrm{I}}$ provides an \textit{a priori} bound on the force generated by a dislocation; this assumption can be easily relaxed by asking $\delta_n^\ell \| \nabla V_{\delta_n} \|_{\infty}$ to be bounded for some $\ell \in \N$.  
Assumption $(V3)_{\textrm{I}}$ identifies a force \textit{barrier} of slope $\tfrac{1}{\delta_n}$ on the boundary of a region of diameter $\delta_n$, which needs to be overcome for a dipole to break apart. In dimension one the region of diameter $\delta_n$ is an interval; in the two-dimensional case, this barrier lives on the boundary of a region $B_n$ contained in the ball $B(0, C \delta_n) \subset \T^2$, namely on a closed curve which encloses $0$. Possible choices for $\Gamma_n$ are spheres of the form $\{ s \in \T^2 : \|s\|_{\Tt} = \gamma \delta_n \}$ or level sets of $V_{\delta_n}$ such as $\{ s \in \T^2 : V_{\delta_n} (s) = - \gamma \log \delta_n \}$. The precise choice of $\Gamma_n$ depends on the type of regularisation and on the symmetries of $V$.

\subsubsection{The one-dimensional case} This is the case we describe in detail. The two-dimensional case will be briefly illustrated in Section~\ref{ss:cvb:C1}, just by highlighting in what it differs from the one-dimensional setup treated here.

We consider a special arrangement of dislocations of alternated sign ($-+-+\dots$) into short dipoles. This corresponds to the choice of $n$ even and $b_i=(-1)^i$ in the energy~\eqref{eqn:disc:energy}.
It is convenient to re-label the positions of positive and negative dislocations as 
\begin{equation*}
  x^{+} = (x_1^{+}, \ldots, x_{\frac{n}{2}}^{+}) \in \T^{\frac{n}{2}}\qquad  \hbox{and}\qquad    x^{-} = (x_1^{-}, \ldots, x_{\frac{n}{2}}^{-}) \in \T^{\frac{n}{2}}\,.
\end{equation*}
We now define a new set of variables $(z,d)\in \T^n$, for $k=1,\dots,\frac n2$, as follows:
\begin{align}\label{variables_zd}
\begin{cases}
\smallskip
z_k = \tfrac12 (x_{k}^+ + x_{k}^-)\\
d_k = \tfrac{1}{2}(x_{k}^+ - x_{k}^-)
\end{cases}
\quad  \Longleftrightarrow \quad 
\begin{cases}
x_{k}^+ = z_k + d_k\\
x_{k}^- = z_k - d_k
\end{cases}.
\end{align}
The variable $z_k$ represents the position of the mid-point of the $k$-th dipole, while $d_k$ is the half-width of the $k$-th dipole. We consider the case of $n/2$ \textit{short} dipoles, where the width of the dipole is much smaller than the dipole-dipole distance. This can be formally expressed by the condition $\max_k |d_k|\ll \min_{i \neq j} |z_i - z_j|_{{\TT}}$.

\smallskip

In terms of the variables $(z,d)$ the energy~\eqref{eqn:disc:energy} reads as 
\begin{align*}
E_n(z,d) = \frac{1}{2n^2} \sum_{k,\ell=1}^{n/2} \sum_{p,q =\pm 1} pq V_{\delta_n}(z_k-z_\ell + pd_k-qd_\ell) - 
\frac{1}{n} \sum_{k=1}^{n/2}\sum_{p=\pm 1} p U(z_k + pd_k),
\end{align*}
and the gradient-flow equation~\eqref{for:evo:dyncs:edge:2D} reads as
\begin{equation}\label{Discrete:gf}
  \frac d{dt} \left[ \begin{array}{c} z \\ d \end{array} \right] 
  = - \frac n2 \left[ \begin{array}{c} \partial_z E_n (z, d) \\ \partial_d E_n (z, d) \end{array} \right].
\end{equation}
Relying on the evenness of $V_{\delta_n}$, we rewrite the right-hand side of \eqref{Discrete:gf} as
\begin{align}\label{pzEn}
\partial_{z_i} E_n(z,d) 
&= \frac1{n^2} \sum_{\ell=1}^{n/2}\sum_{p,q=\pm1} pq V'_{\delta_n}(z_i - z_\ell + p d_i - q d_\ell) - \frac1n\sum_{p=\pm1}p U'(z_i+pd_i), \\\label{pdEn}
\partial_{d_i} E_n(z,d) 
&= \frac1{n^2} \sum_{\ell=1}^{n/2}\sum_{p,q=\pm1} q V'_{\delta_n}(z_i - z_\ell + p d_i - q d_\ell) - \frac1n\sum_{p=\pm1} U'(z_i+pd_i).
\end{align}
Hence, the discrete evolution of the short dipoles is described by the system
\begin{align} \label{f:dtzk}
  \frac d{dt} z_i 
  & = -\frac1{2 n} \sum_{\ell=1}^{n/2}\sum_{p,q=\pm1} pq V'_{\delta_n}(z_i - z_\ell + p d_i - q d_\ell) + \frac12\sum_{p=\pm1}p U'(z_i+pd_i), \\\label{f:dtdk}
  \frac d{dt} d_i 
  & =- \frac1{2 n} \sum_{\ell=1}^{n/2}\sum_{p,q=\pm1} q  V'_{\delta_n} (z_i - z_\ell + p d_i - q d_\ell) + \frac12\sum_{p=\pm1}U'(z_i + p d_i),
\end{align} 
for $i=1,\dots,n/2$.

\medskip

We now introduce a subset of $\T^n$ where we study the evolution; we will refer to it as the \textit{slow manifold}. For any constant $M>0$, we define the set $\Omega(M) \subset \T^n$ as 
\begin{equation}\label{OmegaM1}
\Omega (M) := \left\{ (z,d) \in \T^n: \min_{i \neq j} |z_i - z_j|_{{\TT}} \geq \frac Mn, \ \max_i | d_i | \leq \gamma \delta_n \right\},
\end{equation}
where $\gamma > 0$ is defined in $(V2)_1$.

\begin{thm} \label{t:CvB:1D}
Let $(\delta_n)_{n\in \N}\subset (0,\infty)$ be a sequence such that $n^3 \delta_n \to 0$ as $n\to \infty$. Let $U$, $V$ and $V_{\delta_n}$ satisfy conditions $(V1)_{\textrm{I}}$-$(V5)_{\textrm{I}}$ and $(U)_{\textrm{I}}$. Let $A, T > 0$ be given constants. 
Then, for every initial condition $(z^n_\circ, d^n_\circ) \in \Omega (2A)$, the solution $(z^n(t),d^n(t))$ of \eqref{f:dtzk}-\eqref{f:dtdk} satisfies $(z^n(t),d^n(t))\in \Omega (A)$ for every $t\in [0,T]$ and for large enough $n$. Moreover, the empirical measures $\mu_n^\pm (t)$ and $\mu_{\circ,n}^\pm$ associated to $x^{n,\pm} (t)$ and to $x_\circ^{n,\pm}$, respectively, where $x^{n,\pm} (t)$ and $x_\circ^{n,\pm}$ are related to $(z^n(t),d^n(t))$ and $(z^n_\circ, d^n_\circ)$ via \eqref{variables_zd}, satisfy $\mu_n^\pm (t) - \mu_{\circ,n}^\pm \weakto 0$ narrowly as $n \to \infty$, uniformly in $t \in [0, T]$.
\end{thm}

From here on, to ease the notation, we will drop the explicit dependence on $n$ when it is clear from the context.

\begin{proof}
The main idea of the proof is that the evolution of $z$ is \textit{slow} on the manifold, namely the forces $\tfrac n2 \partial_z E_n(z,d)$ acting on $z$ are uniformly small. The forces on $d$ may instead be large, but point in the `right direction', i.e., there is a large basin of attraction to some $\tilde d(t)$ which is close to $0$.
\smallskip

The above  argument will be made precise by means of three main steps. To this aim, we define for $M > 0$
\begin{equation*}
\Omega_2 (M) 
:= \left\{ (z,d) \in \T^n: \min_{i \neq j} |z_i - z_j|_{{\TT}} \geq \frac Mn, \ \max_i | d_i | \leq 2 \gamma \delta_n \right\} 
\supset
\Omega (M).
\end{equation*}
In all our estimates the symbols $C, c$ denote positive constants which only depend on $U$, $V$, $A$ and the $n$-independent constants appearing in $(V2)_{\textrm{I}}$--$(V5)_{\textrm{I}}$.
\smallskip

\noindent
\textbf{Step 1: Behaviour of the gradient-flow equations \eqref{f:dtzk}-\eqref{f:dtdk} in $\Omega_2 (A)$.} 
In this step we prove that the right-hand side of \eqref{f:dtzk} satisfies the bound
\begin{equation} \label{estimate:slowz}
  \exists \, C > 0 : 
  \sup_{(z,d) \in \Omega_2 (A)} \| \tfrac n2 \,\partial_z E_n(z,d) \|_\infty \leq C n^2\delta_n,
  \text{ for all $n$ large enough};
\end{equation}
additionally, we rewrite the evolution of $d_i$ on $\Omega_2 (A)$ in a more convenient form.
\smallskip

To prove \eqref{estimate:slowz}, we note that by the symmetry of $V_{\delta_n}$ the term $\ell = i$ in the first sum in the right-hand side of \eqref{f:dtzk} is zero; for $\ell\neq i$, we estimate the summand by using the Mean Value Theorem on the sum over $q$:
\begin{align} \label{fp:bd:dtzk}
\bigg| \sum_{p=\pm1} p \sum_{q=\pm1} q V'_{\delta_n}(z_i - z_\ell + p d_i - q d_\ell)\bigg| 
= \bigg| \sum_{p=\pm1} 2 p d_\ell V''_{\delta_n}(\alpha_{i\ell}) \bigg| 
\leq 4 \|d\|_{\infty} |V''_{\delta_n}(\alpha_{i\ell})|,
\end{align}
where $|\alpha_{i\ell}-(z_i-z_\ell)|_{{\TT}} \leq 2 \|d\|_{\infty} \leq C \delta_n$. Hence, $\min_{i \neq \ell} |\alpha_{i\ell}|_{{\TT}}  \geq \frac{M}{2n}$ for all $n$ large enough. By the assumption $(V5)_{\textrm{I}}$ we then conclude from \eqref{fp:bd:dtzk} that 
$$
\bigg|\sum_{p,q=\pm1} pq V'_{\delta_n}(z_i - z_\ell + p d_i - q d_\ell)\bigg| \leq C n^2 \delta_n.
$$
For the forcing term (the second sum in the right-hand side of \eqref{f:dtzk}), by $(U)$ we estimate
\begin{align*}
|U'(z_i + d_i) - U'(z_i - d_i)| 
\leq 2 \| d \|_\infty \| U'' \|_{L^\infty(\T)}
\leq C \delta_n.
\end{align*}
This completes the proof of \eqref{estimate:slowz}.
\smallskip

Now we rewrite the evolution of the dipole width $d_i$, under the assumption that $(z,d) \in \Omega_2 (A)$ during the evolution. We claim that, in \eqref{f:dtdk}, the term corresponding to $\ell = i$ is the dominant term. By isolating this term, we rewrite the evolution of $d_i$ as
\begin{align} \label{Taylorode}
 \frac d{dt} d_i 
 = \frac1n V'_{\delta_n}(2d_i) + \frac1{2 n} \sum_{\substack{\ell=1\\ \ell\neq i}}^{n/2}\sum_{p,q=\pm1} q (- V'_{\delta_n} (z_i - z_\ell + p d_i - q d_\ell) )+ \frac12\sum_{p=\pm1}U'(z_i + p d_i),
\end{align} 
and the computations leading to the estimate \eqref{estimate:slowz} show that $|F_n^i(z,d)|\leq C$ for all $i$ and for large enough~$n$, where 
$$
F_n^i(z,d):= \frac1{2 n} \sum_{\substack{\ell=1\\ \ell\neq i}}^{n/2}\sum_{p,q=\pm1} q (- V'_{\delta_n} (z_i - z_\ell + p d_i - q d_\ell) )+ \frac12\sum_{p=\pm1}U'(z_i + p d_i).
$$
\smallskip

\noindent
\textbf{Step 2: The evolution \eqref{f:dtzk}-\eqref{f:dtdk} remains in the manifold~$\Omega(A)$.}
We claim that the solution $(z(t), d(t))$ of \eqref{f:dtzk}-\eqref{f:dtdk} with initial condition $(z_\circ, d_\circ) \in \Omega (2A)$ satisfies  $(z(t), d(t)) \in \Omega (A)$ for every $t\in [0,T]$, for large enough $n$. We first prove with \emph{a priori} bounds that $(z(t), d(t)) \in \Omega_2 (A)$ for every $t\in [0,\delta_n^3]$. Then, we use Step 1 to improve these bounds and show that $(z(t),d(t)) \in \Omega (2A - 2C n^3 \delta_n^4) \subset \Omega (A)$. In the final part we iterate this procedure in time to obtain the desired result on $[0, T]$.
\smallskip

\noindent
\textit{Step 2.1: The evolution with initial data in $\Omega (B)$, $B > A$, stays in $\Omega_2 (A)$ for small time.}
We claim that the solution $(\tilde z(t), \tilde d(t))$ of \eqref{f:dtzk}-\eqref{f:dtdk} with initial condition $(\tilde z_\circ, \tilde d_\circ) \in \Omega (B)$, with $B>A$, satisfies  $(\tilde z(t), \tilde d(t)) \in \Omega_2 (A)$ for  all $t\in [0, \delta_n^3]$, and for large enough $n$.
By applying the bound $(V4)_{\textrm{I}}$ for every term in the sum in the right-hand side of \eqref{pzEn}-\eqref{pdEn}, we obtain the \textit{a priori} estimate
$$
\| n \nabla E_n\|_{\infty} \leq C / \delta_n \text{ for all $n$ large enough}.
$$
Hence, the right-hand side of the gradient-flow evolutions \eqref{f:dtzk}-\eqref{f:dtdk} is bounded uniformly by $C / \delta_n$. Thus, for $t\in [0,\delta_n^3]$ we have that, since $(\tilde z_\circ, \tilde d_\circ) \in \Omega (B)$,
\begin{equation*}
  \|\tilde d(t)\|_\infty 
\leq \|\tilde d_\circ\|_\infty + C t / \delta_n  
\leq 2 \gamma \delta_n
\text{ for all $n$ large enough}.
\end{equation*}
A similar estimate for $\tilde z(t)$ allows us to deduce that 
\begin{equation*}
  \min_{i \neq j} |\tilde z_i(t) - \tilde z_j(t)|_{{\TT}} 
  \geq \min_{i \neq j} |\tilde z_{\circ,i} - \tilde z_{\circ,j}|_{{\TT}} - 2 \frac C{\delta_n} t
  \geq \frac{ B - 2 C n \delta_n^2 }n.
\end{equation*}
Hence, for $n$ large enough, $(\tilde z(t), \tilde d(t)) \in \Omega_2 (A)$ for every $t\in [0,\delta_n^3]$.

\smallskip

\noindent
\textit{Step 2.2: Improved estimates for $(\tilde z(t), \tilde d(t))$ for $t\in [0, \delta_n^3]$.} 
Since $(\tilde z(t), \tilde d(t)) \in \Omega_2 (A)$ for every $t\in [0,\delta_n^3]$, we obtain from \eqref{estimate:slowz} the improved estimate 
\begin{equation} \label{zt:est}
  \|\tilde z(t) - \tilde z_\circ\|_\infty \leq C n^2 \delta_n t \leq C n^2 \delta^4_n 
\end{equation}
which leads to
\begin{equation} \label{zij:est}
  \min_{i \neq j} |\tilde z_i(t) - \tilde z_j(t)|_{{\TT}}  
  \geq \frac{ B - 2C n^3\delta_n^4 }n \\ 
  \text{ for all $t\in [0,\delta_n^3]$ and all $n$ large enough}.  
\end{equation}
Moreover, we claim that 
\begin{equation} \label{dinf:est}
  \| \tilde d(t) \|_\infty \leq \gamma \delta_n
  \quad \text{for all $t\in [0,\delta_n^3]$ and all $n$ large enough}.
\end{equation}
Indeed, $\| \tilde d_\circ \|_\infty \leq \gamma \delta_n$, and by \eqref{Taylorode}  we have 
\begin{equation*}
 \frac d{dt} \tilde d_i = \frac1n V'_{\delta_n}(2 \tilde d_i) + F_n^i(\tilde z, \tilde d),
\end{equation*}
where, by Step 2.1, $|F_n^i( \tilde z (t), \tilde d (t) )| \leq C$ for all $t \in [0, \delta_n^3]$. We conclude by contradiction; suppose $\| \tilde d(t) \|_\infty > \gamma \delta_n$. Then, since $\tilde d \in C([0, \delta_n^3]; \T^{n/2})$, there exist $i \in \{1, \ldots, \tfrac n2\}$ and $t \in [0, \delta_n^3]$ such that $\tilde d_i (t) = p \gamma \delta_n$ and $p \tfrac d{dt} \tilde d_i (t) \geq 0$ for some $p \in \{-1, +1\}$. However, by \eqref{Taylorode} and $(V3)_{\textrm{I}}$, 
\begin{equation*}
  p \frac d{dt} \tilde d_i (t)
  = \frac pn V'_{\delta_n}(2 p \gamma \delta_n) + p F_n^i(\tilde z (t), \tilde d(t))
  \leq - \frac c{ n \delta_n } + C,
\end{equation*}
which is negative for all $n$ large enough. We conclude that \eqref{dinf:est} holds.

In conclusion, for every $t\in [0,\delta_n^3]$ we have that $(\tilde z(t), \tilde d(t)) \in \Omega (B - 2C n^3\delta_n^4)$.

\smallskip

\noindent
\textit{Step 2.3: Iteration.} We iterate Step 2.2 in the time intervals $[\ell \delta_n^3, (\ell +1) \delta_n^3]$ for $\ell =1,\dots, \left\lfloor  T / \delta_n^3 \right\rfloor$ to construct the solution $(z(t), d(t))$ with initial condition $(z_\circ, d_\circ) \in \Omega (2A)$. The corresponding values for $B$ are
\begin{equation*}
B = 2 (A - C (\ell + 1) n^3 \delta_n^4) 
\quad \text{for all } t\in (\ell \delta_n^3, (\ell +1) \delta_n^3].
\end{equation*}
Since $\ell \leq T / \delta_n^3$, we obtain $B \geq 2 (A - C T n^3 \delta_n) > A$ for all $n$ large enough. Hence, it follows from \eqref{zt:est}, \eqref{zij:est} and \eqref{dinf:est} that
\begin{equation} \label{zdt:nice}
\| z(t) - z_\circ\|_\infty \leq C n^2 \delta_n
\quad \text{ and } \quad
(z(t), d(t)) \in \Omega (A)
\end{equation}
 for all $t\in [0,T]$ and all $n$ large enough.

\medskip
\noindent
\textbf{Step 3: $\mu_n^\pm (t) - \mu_{\circ,n}^\pm \weakto 0$ as $n \to \infty$, uniformly in $[0, T]$.}
Let $\varphi \in C(\T)$ be an arbitrary test function, and let $\omega:[0,\infty)\to[0,\infty)$ be its modulus of continuity. Using \eqref{variables_zd}, we estimate
\begin{multline*}
  \bigg| \int_{\T} \varphi d \mu_n^+ (t) - \int_{\T} \varphi d \mu_{\circ,n}^+ \bigg|
  \leq \frac1n \sum_{i = 1}^{n/2} \big| \varphi \big( z_i(t) + d_i(t) \big) - \varphi ( z_{\circ, i} +  d_{\circ, i} ) \big| \\
  \leq \omega \big( \| z(t) - z_\circ \|_\infty + \| d(t) \|_\infty + \| d_\circ \|_\infty \big).
\end{multline*}
From \eqref{zdt:nice} we conclude that the right-hand side tends to $0$ as $n \to \infty$ pointwise for all $t \in [0,T]$. Similarly we obtain the convergence of $\mu_n^- (t) - \mu_{\circ,n}^-\weakto 0$ as $n \to \infty$, uniformly in $[0, T]$.
\end{proof}

\begin{cor} \label{c:CvB:1D}
Let $\delta_n$, $U$, $V$, $V_{\delta_n}$ and $T > 0$ be as in Theorem \ref{t:CvB:1D}. Let $\rho_\circ \in  \mathcal P(\T)\cap L^\infty (\T)$, and let $\mathcal{A}_\circ$ be the set of approximating sequences defined as 
\begin{align*}
\mathcal{A}_\circ:=\big\{(x_\circ^n,b_\circ^n)_{n\in \N}:& (x_\circ^n,b_\circ^n)\in \T^n\times \{\pm 1\}^n \ \textrm{for every } n\in \N, \exists \ A>0 \textrm{ such that }\\ 
&\quad (z_\circ^n, d_\circ^n) \in \Omega (2A) \textrm{ for large } n, \textrm{and } \mu_{\circ,n}^\pm \weakto \rho_\circ \textrm{ as } n\to \infty\big\},
\end{align*}
where $(z_{\circ}^n, d_{\circ}^n)$ and $\mu_{\circ,n}^\pm$ are defined in terms of $(x_{\circ}^n,b_\circ^n)$ via \eqref{variables_zd} and \eqref{fd:mun}, respectively. 
Then $\mathcal{A}_\circ \neq \emptyset$. Now assume that $U' \not \equiv 0$ on $\supp \rho_\circ$. Let $(x_{\circ}^n,b_\circ^n)\in \mathcal{A}_\circ$, and denote with $\mu_{n}^\pm (t)$ the empirical measures associated to the solution $x^n(t)$ of \eqref{for:evo:dyncs:edge:2D} with initial datum $(x_{\circ}^n, b_{\circ}^n)$. Then, the weak limit of $\mu_{n}^\pm (t)$, namely the measure $\brho = (\rho^+,\rho^-)$ such that $\mu_{n}^\pm (t)  \weakto \rho^\pm (t)$ as $n\to \infty$ for a.e.~$t \in (0,T)$, is not a solution of \eqref{f:GB-repeated} with initial datum $\brho_\circ=(\rho_\circ, \rho_\circ)$ in $(0,T)$.
\end{cor}

\begin{proof} 
It is easy to see that $\brho_\circ = (\rho_\circ, \rho_\circ)$ is not a stationary solution to \eqref{f:GB-repeated}. Thus it suffices to show that $\mu_{n}^\pm (t) - \mu_{\circ,n}^\pm \weakto 0$ as $n \to \infty$. This property is guaranteed by Theorem \ref{t:CvB:1D}.

It remains to show that $\mathcal{A}_\circ \neq \emptyset$. We set $x_{\circ, 1}^{n,+} := 0$, and we choose the positions of the other particles iteratively, so that
\begin{equation*}
  \int_{x_{\circ, i}^{n,+}}^{x_{\circ, i+1}^{n,+}} \rho_\circ^+(x)\, dx = \frac2n
  \quad \text{for all } i = 1,\ldots, \frac n2 - 1.
\end{equation*}
Then, we set $x_{\circ}^{n,-} := x_{\circ}^{n,+}$. By construction, $\mu_{\circ,n}^\pm \weakto \rho_\circ^\pm$ as $n \to \infty$, and $z_\circ^n = x_{\circ}^{n,+}$ and $d_\circ^n = 0$.
Since $\rho_\circ \in L^\infty (\T)$, we obtain that $\min_i |z_{\circ,i+1}^n - z_{\circ,i}^n |_{{\TT}} \geq (n \| \rho_\circ \|_\infty )^{-1}$. Using that $z_\circ^n$ is ordered, we conclude that $(z_{\circ}^n, d_{\circ}^n) \in \Omega ( 1 / \| \rho_\circ \|_\infty )$.
\end{proof}

\begin{rem}[Sharper estimates under higher regularity of $V$]
\label{ss:cvb:1D:sharp}
A careful inspection of the proof of Theorem \ref{t:CvB:1D} shows that the result holds true under the weaker assumption $n^{3/2} delta_n$,
provided $(V2)_{\textrm{I}}$ and $(V5)_{\textrm{I}}$ are replaced by the stronger conditions
\begin{itemize}
\item[$(V2')_{\textrm{I}}$] $V_{\delta_n} \in C^3(\T)$ is even;
\item[$(V5')_{\textrm{I}}$]  $\forall \, c > 0 \ \exists \, C > 0 \ \forall \, |s|_{\TT} > c/n : | V^{(3)}_{\delta_n} (s) | \leq C / s^3$.
\end{itemize}
Note that $(V5')_{\textrm{I}}$ is satisfied by our model example of $V$, so it reasonable to assume that it is satisfied by the regularised potential $V_{\delta_n}$ as well.

In \eqref{fp:bd:dtzk} we can interpret the sum over $p$ and $q$ as a four-point approximation scheme of the third derivative $V^{(3)}_{\delta_n} (z_i - z_\ell)$, and since $|z_i- z_\ell|_{{\TT}} \geq c |\tfrac{i - \ell}n |$, by $(V5')_{\textrm{I}}$ the bound in \eqref{fp:bd:dtzk} results in 
\begin{align} \label{RHSz:sharp}
\bigg| \sum_{p,q=\pm1} p q V'_{\delta_n}(z_i - z_\ell + p d_i - q d_\ell)\bigg| 
\leq C \|d\|_{\infty}^2 \max_{|s|_{{\TT}} \geq c \big|\tfrac{i-\ell}n\big| } \big| V^{(3)}_{\delta_n}(s) \big|
\leq C \frac{ n^3 \delta_n^2 }{|i - \ell|^3}.
\end{align}
By using \eqref{RHSz:sharp} for each term $\ell \neq i$ in the first sum in \eqref{f:dtzk}, we get the improved estimate 
 \begin{equation}\label{RHSz:sharp2}
   \sup_{(z,d) \in \Omega_2 (A)} \| \tfrac n2 \,\partial_z E_n(z,d) \|_\infty \leq C n^2 \delta_n^2
 \end{equation}
 instead of \eqref{estimate:slowz}. Proceeding as in the proof of Theorem \ref{t:CvB:1D} and by using the improved estimate \eqref{RHSz:sharp2}, we obtain
\begin{equation*} 
\| \tilde z(t) - \tilde z_\circ\|_\infty \leq C n^2 \delta^2_n
\quad \text{ and } \quad
(z(t), d(t)) \in \Omega (A) 
\end{equation*}
for all $t\in [0,T]$ and all $n$ large enough, instead of \eqref{zdt:nice}, provided $n^{3/2} \delta_n \to 0$.
\end{rem}

\subsubsection{The two-dimensional case}
\label{ss:cvb:C1}
In this short section we only stress the differences with the one-dimensional case considered above.

First of all, we introduce the subset of $(\T^n)^2$ where we study the evolution, namely the \textit{slow manifold}. In analogy with \eqref{OmegaM1}, for any constant $M>0$, we define the set 
$\Omega(M)\subset (\T^2)^n$ as 
\begin{equation*}
\Omega (M) := \left\{ (z,d) \in (\T^2)^n: \min_{i \neq j} \|z_i - z_j\|_{\Tt} \geq \frac{M}{\sqrt n}, \,\, d_i \in B_n \,\, \forall \, i \right\},
\end{equation*}
where $B_n$ is the region whose boundary is the `trapping' curve $\Gamma_n$ in $(V3)_{\textrm{I}}$.

\smallskip

The following theorem is the counterpart of Theorem \ref{t:CvB:1D} and Corollary \ref{c:CvB:1D} for $d=2$.

\begin{thm} \label{t:CvB:2D}
Let $(\delta_n)_{n\in \N}\subset (0,\infty)$ be a sequence such that $n^2 \delta_n \to 0$ as $n\to \infty$. Let $U$, $V$ and $V_{\delta_n}$ satisfy conditions $(V1)_{\textrm{I}}$-$(V5)_{\textrm{I}}$ and $(U)_{\textrm{I}}$, and let $T>0$ be fixed. Let $\rho_\circ \in  \mathcal P(\T^2)\cap L^\infty (\T^2)$, and let $\mathcal{A}_\circ$ be the set of approximating sequences defined as 
\begin{align*}
\mathcal{A}_\circ:=\big\{(x_\circ^n,b_\circ^n)_{n\in \N}:& (x_\circ^n,b_\circ^n)\in (\T^2)^n\times \{\pm 1\}^n \ \textrm{for every } n\in \N, \exists \ A>0 \textrm{ such that }\\ 
&\quad (z_\circ^n, d_\circ^n) \in \Omega (2A) \textrm{ for large } n, \textrm{and } \mu_{\circ,n}^\pm \weakto \rho_\circ \textrm{ as } n\to \infty\big\},
\end{align*}
where $(z_{\circ}^n, d_{\circ}^n)$ and $\mu_{\circ,n}^\pm$ are defined in terms of $(x_{\circ}^n,b_\circ^n)$ via \eqref{variables_zd} and \eqref{fd:mun}, respectively. 
Then $\mathcal{A}_\circ \neq \emptyset$. Now assume that $\nabla U \not \equiv 0$ on $\supp \rho_\circ$. Let $(x_{\circ}^n,b_\circ^n)\in \mathcal{A}_\circ$, and denote with $\mu_{n}^\pm (t)$ the empirical measures associated to the solution $x^n(t)$ of \eqref{for:evo:dyncs:edge:2D} with initial datum $(x_{\circ}^n, b_{\circ}^n)$. Then, the weak limit of $\mu_{n}^\pm (t)$, namely the measure $\brho = (\rho^+,\rho^-)$ such that $\mu_{n}^\pm (t)  \weakto \rho^\pm (t)$ as $n\to \infty$ for a.e.~$t \in (0,T)$, is not a solution of \eqref{f:GB-repeated} with initial datum $\brho_\circ=(\rho_\circ, \rho_\circ)$ in $(0,T)$.
\end{thm}

\begin{proof}
The proof of Theorem \ref{t:CvB:2D} follows by a straightforward adaptation of the proof of Theorem \ref{t:CvB:1D} to the two-dimensional setting. 

Note that the estimate \eqref{estimate:slowz} holds true, in the two-dimensional setting, with the bound $C n \delta_n$ (instead of $C n^2 \delta_n$), and this leads to the weaker condition $n^2 \delta_n \to 0$ for $\delta_n$.
\end{proof}

\medskip

\begin{rem}[Sharper estimates under higher regularity]
The extension discussed in Remark~\ref{ss:cvb:1D:sharp} has a two-dimensional equivalent. The bound on $D^3 V_{\delta_n}$ provides again an additional factor of $\delta_n$ in the right-hand side of \eqref{estimate:slowz}. To use the condition $\min_{i \neq j} \|z_i - z_j\|_{\Tt} \geq c / \sqrt n$, let us fix the index $i$, and relabel $z_\ell$ such that $\|z_i - z_\ell\|_{\Tt}$ is increasing in $\ell$ for $\ell \neq i$. Since the balls $B(z_\ell, c/(2 \sqrt n))$ are disjoint, for all $R \geq c/(2 \sqrt n)$ the ball $B(z_i, R)$ contains at most $\lfloor CnR^2 \rfloor$ points $z_\ell$. In other words, there are at most $\lfloor CnR^2 \rfloor$ points $z_\ell$ such that $\|z_i - z_\ell\|_{\Tt} \leq R$. 
Hence, $\|z_i - z_\ell\|_{\Tt} \geq c'' \sqrt{ \ell/n }$ for some $c'' > 0$. We use this estimate to  replace \eqref{RHSz:sharp} by
\begin{align*} 
\bigg| \sum_{p,q=\pm1} p q V'_{\delta_n}(z_i - z_\ell + p d_i - q d_\ell)\bigg| 
\leq C \delta_n^2 \Big( \frac{n}{\ell} \Big)^{\tfrac32},
\end{align*}
Then, \eqref{estimate:slowz} becomes
 \begin{equation*}
   \sup_{(z,d) \in \Omega_2 (A)} \| \tfrac n2 \,\partial_z E_n(z,d) \|_\infty \leq C \delta_n ( \tfrac1n + \sqrt n \, \delta_n ),
 \end{equation*}
which leads to the weaker condition on $\delta_n$ given by $n^{5/4} \delta_n \to 0$. 
\end{rem}

\subsection{The two-dimensional case, slip-place-confined motion} \label{ss:cvb:C2}
This setting is the one inspired by the case of edge dislocations, whose interaction potential in $\R^2$ is $V_{\text{edge}} (x) = -\log |x| + (e_1 \cdot x / |x|)^2$ (see, e.g., (5-16) in \cite{HirthLothe82}) as discussed in Section~\ref{subsec:description-confined}. It is also the case corresponding exactly to the Groma-Balogh evolution equations.
\smallskip

Given some $\delta_n > 0$, we define the slow manifold as
\begin{equation*}
\Omega (M) := \left\{ (z,d) \in (\T^2)^n: \min_{i \neq j} \|z_i - z_j\|_{\Tt} \geq \frac{M}{\sqrt n}, \,\, \max_i | d_i | < \delta_n \,\, \text{and } \, \min_i |d_i\cdot e_2| > 0 \right\}.
\end{equation*}

\smallskip

The following theorem is the counterpart of Theorem \ref{t:CvB:1D} and Corollary \ref{c:CvB:1D} for a slip-confined evolution in dimension $d=2$. Note that, unlike in Section~\ref{ss:cvb:1D}, we do not consider a regularised potential $V_{\delta_n}$, but we deal directly with $V$.

\begin{thm} \label{t:CvB:2D-confined}
Let $(\delta_n)_{n\in \N}\subset (0,\infty)$ be a sequence such that $n^2 \delta_n \to 0$ as $n\to \infty$. Let $U$ satisfy $(U)_{\textrm{I}}$, let $V$ be as in Case 2, and let $T > 0$ be fixed. Let $\rho_\circ \in \mathcal P(\T^2)\cap L^\infty (\T^2)$, and let $\mathcal{A}_\circ$ be the set of approximating sequences defined as 
\begin{align*}
\mathcal{A}_\circ:=\big\{(x_\circ^n,b_\circ^n)_{n\in \N}:& (x_\circ^n,b_\circ^n)\in (\T^2)^n\times \{\pm 1\}^n \ \textrm{for every } n\in \N, \exists \ A>0 \textrm{ such that }\\ 
&\quad (z_\circ^n, d_\circ^n) \in \Omega (2A) \textrm{ for large } n, \textrm{and } \mu_{\circ,n}^\pm \weakto \rho_\circ \textrm{ as } n\to \infty\big\},
\end{align*}
where $(z_{\circ}^n, d_{\circ}^n)$ and $\mu_{\circ,n}^\pm$ are defined in terms of $(x_{\circ}^n,b_\circ^n)$ via \eqref{variables_zd} and \eqref{fd:mun}, respectively. 
Then $\mathcal{A}_\circ \neq \emptyset$. Now assume that $\partial_1 U \not \equiv 0$ on $\supp \brho_\circ$. Let $(x_{\circ}^n,b_\circ^n)\in \mathcal{A}_\circ$, and denote with $\mu_{n}^\pm (t)$ the empirical measures associated to the solution $x^n(t)$ of the unregularised evolution \eqref{for:evo:dyncs:edge:2D-confined}, with initial datum $(x_{\circ}^n, b_{\circ}^n)$. Then, the weak limit of $\mu_{n}^\pm (t)$, namely the measure $\brho = (\rho^+,\rho^-)$ such that $\mu_{n}^\pm (t)  \weakto \rho^\pm (t)$ as $n\to \infty$ for a.e.~$t \in (0,T)$, is not a solution of \eqref{f:GB-slip-plane-confined} with initial datum $\brho_\circ=(\rho_\circ, \rho_\circ)$ in $(0,T)$.
\end{thm}

\begin{proof} We start by establishing some properties of $V$ that will enable us to proceed similarly as in the proof of Theorem \ref{t:CvB:1D}. 

Note that $V$ `almost' satisfies assumptions $(V1)_{\textrm{I}}$--$(V5)_{\textrm{I}}$. Indeed it does satisfy $(V1)_{\textrm{I}}$ and $(V5)_{\textrm{I}}$; moreover, $V \in C^2(\T^2 \setminus \{0\})$ and is even, which gives almost $(V2)_{\textrm{I}}$. Instead of $(V3)_{\textrm{I}}$, we prove a sufficient upper bound on $\partial_1 V$ around $0$.  With this aim, we observe from \eqref{VR2eqn} and \eqref{VT2eqnOnR2} that $W := V - V_{\text{edge}}$ is biharmonic on $(-1,1)^2$, and thus smooth on the closed square $Q_{1/2}:=[-1/2,1/2]^2$. Hence, for all $\delta$ small enough,
$$ 
\partial_1 V (2 \delta, h) \leq - \frac 6{25 \delta} + 
\| \partial_1 W \|_{L^\infty (Q_{1/2})} \quad \text{for all } |h| \leq \delta.
$$
Finally, instead of $(V4)_{\textrm{I}}$, we have that $s V(s)$ is bounded in $\T^2$.

Given any $(z^n, d^n) \in \Omega (M)$, the singular potential $V$ is never evaluated on $B(0, r_n)$, where 
\begin{equation*}
  r_n := \min_i |d_i \cdot e_2| > 0.
\end{equation*}
Hence the extension of Theorem \ref{t:CvB:1D} follows similarly as in Section~\ref{ss:cvb:C1}, with $V_{\delta_n}$ replaced by $V$. In particular, we note that the equivalent of the a priori estimates in Step 2.1 depends on $r_n$. Hence, the time interval $[0, \delta_n^3]$ may need to be shrunk when $r_n > 0$ is small. Nevertheless, thanks to the $r_n$-independently improved estimate in Step 2.2, the proof is easily adjusted to smaller time intervals.
\end{proof}

\begin{rem}
We note that with minor modifications to the proof it is possible to remove in Theorem \ref{t:CvB:2D-confined} the lower bound on $|d_i \cdot e_2|$ in $\Omega (M)$ (i.e., to allow the two dislocations in a dipole to reside on the same slip plane) at the cost of introducing a regularisation $V_{\delta_n}$ similarly as in Section~\ref{ss:cvb:C1}. 
\end{rem}

%\appendix
\renewcommand{\thesection}{A}

\section{Appendix: Orlicz spaces and embeddings}\label{sect:Orlicz}

\renewcommand{\theequation}{A.\arabic{equation}}
\renewcommand{\thethm}{A.\arabic{thm}}

In this section we recall the definition and some properties of Orlicz spaces that we have used throughout the paper. This functional framework was used by~\cite{CannoneEl-HajjMonneauRibaud10} to prove existence results for the Groma-Balogh equations~\eqref{f:GB-slip-plane-confined}.

\medskip

For brevity, we focus on  real-valued Lebesgue-measurable functions on the flat torus $\T^d$, for any $d \geq 1$. 
First we introduce Young functions (see~\cite[Sec.~3.1]{RaoRen-O} or~\cite[Sec.~3.1]{CannoneEl-HajjMonneauRibaud10}). A function $\phi: [0,+\infty) \to [0,+\infty] $ is a Young function if $\phi$ is continuous, convex, $\phi(0) = 0$, and $\lim_{t\to\infty}\phi(t)/t = +\infty$. The Orlicz class $K_\phi (\T^d)$ is the set of (equivalence classes of) measurable functions $g : \T^d \to \R$ satisfying $\int_{\T^d} \phi(|g(x)|)dx < \infty$. The Orlicz space $L_\phi (\T^d)$ is the linear hull of $K_\phi (\T^d)$ equipped with the Luxemburg norm
\begin{equation} \label{O:norm}
  \| g \|_{L_\phi (\T^d)}:= \inf \bigg\{ \lambda > 0 : \int_{\T^d} \phi \bigg( \frac{|g(x)|}\lambda \bigg)dx \leq 1 \bigg\},
\end{equation}
and is a Banach space. In general, $(L_\phi (\T^d), \| \cdot \|_{L_\phi (\T^d)})$ is neither separable nor reflexive. However, the closure in $L_\phi (\T^d)$ of bounded functions, denoted with $E_\phi (\T^d)$, is separable, and $E_\phi (\T^d)\subseteq K_\phi (\T^d) \subseteq L_\phi (\T^d)$.

For the choices 
\begin{equation*}
  \phi_\alpha (t) := e^{t^\alpha} - 1 \ \text{ with } \alpha \geq 1, 
  \quad \text{and} \quad 
  \phi^\beta (t) := t (\log(e + t))^\beta  \ \text{ with }  \beta \geq 0,
\end{equation*}
we denote
\begin{equation*}
  \Exp_\alpha (\T^d) := L_{\phi_\alpha} (\T^d), \quad \EExp_\alpha (\T^d) := E_{\phi_\alpha} (\T^d)
  \quad \text{and} \quad
  L \log^\beta L (\T^d) := L_{\phi^\beta} (\T^d).
\end{equation*}
It is easy to see that these Orlicz spaces are ordered, i.e., $\Exp_{\alpha_2} (\T^d) \subseteq \Exp_{\alpha_1} (\T^d)$ for all $1 \leq \alpha_1 \leq \alpha_2$, and $L \log^{\beta_2} L (\T^d) \subseteq L \log^{\beta_1} L (\T^d)$ for all $0 \leq \beta_1 \leq \beta_2$. For convenience, we set
\begin{equation*}
  \Exp (\T^d) := \Exp_1 (\T^d),
   \quad
  L \log L (\T^d) := L \log^1 L (\T^d), \quad \text{and} \quad  \EExp (\T^d) := \EExp_1 (\T^d).
\end{equation*}

\medskip

\noindent
Finally, we recall the definition of the fractional Sobolev space $H^s(\R^d)$. For $s \geq 0$, we set
\begin{equation*}
  H^s (\R^d)
  = \{ u \in L^2(\R^d) : [u]_{H^s (\R^d)} < \infty \},
  \quad [u]_{H^s (\R^d)} := \int_{\R^d} |\xi|^{2s} | \mathcal F u (\xi) |^2 \, d\xi,
\end{equation*}
where $\mathcal F$ is the Fourier transform on $\R^d$. We note that the Gagliardo (semi)norm $[ \cdot ]_{H^s (\R^d)}$ is related to the usual norm on $H^s (\R^d)$ given by
$\| u \|_{H^s (\R^d)}^2 = \| u \|_{L^2 (\R^d)}^2 + [ u ]_{H^s (\R^d)}^2$. We will only be interested in the compact embedding of fractional Sobolev spaces into Orlicz spaces. For a more complete treatment of fractional Sobolev spaces we refer to \cite{DiNezzaPalatucciValdinoci12}.

\begin{lem}[Properties of $\Exp_\alpha (\T^d)$ and $L \log^\beta L (\T^d)$] \label{l:Orlicz} 
Let $d \geq 1$. The following properties are satisfied:
\begin{enumerate}[label=(\roman{*})]
    \item \label{l:O:bdd} for every $C>0$, the sublevel set $\{f\in L^1(\T^d): \Ent(|f|)\leq C\}$ is bounded in $L\log L(\T^d)$;
  \item \label{l:O:LlogL:dual} $\big( L \log^{\beta} L (\T^d) \big)^* = \Exp_{1/\beta} (\T^d)$ for all $\beta>0$; 
  \item \label{l:O:Exp:dual} $ \big( \EExp_\alpha (\T^d) \big)^* = L \log^{1/\alpha} L (\T^d)$ for all $\alpha>0$; 
   \item \label{l:O:Hol:L1} there exists a constant $C>0$ such that $\| f g \|_{L^1 (\T^d)} \leq C \| f \|_{L \log L(\T^d)} \| g \|_{\Exp (\T^d)}$ for all $f \in L \log L(\T^d)$ and all $g \in \Exp (\T^d)$;
   \item \label{l:O:Hol:LlogL} there exists a constant $C>0$ such that $\| f g \|_{L \log^{1/2} L(\T^d)} \leq C \| f \|_{L \log L(\T^d)} \| g \|_{\Exp_2 (\T^d)}$ for all $f \in L \log L(\T^d)$ and all $g \in \Exp_2(\T^d)$;
  \item \label{l:O:H1:Exp:ct} $H^{d/2} (\T^d) \hookrightarrow \Exp_2 (\T^d) \hookrightarrow \Exp_\alpha (\T^d) \hookrightarrow \EExp (\T^d) \hookrightarrow \Exp (\T^d)$ for all $1 < \alpha  < 2$.
  \item \label{l:O:H1:Exp:cp} $H^{d/2} (\T^d) \subset \subset \Exp_\alpha (\T^d)$ for all $1 < \alpha  < 2$.
\end{enumerate}

\end{lem}

\begin{proof} 
Property \ref{l:O:bdd} follows from \eqref{O:norm} by elementary estimates. 
\smallskip

\noindent
Properties \ref{l:O:LlogL:dual} and \ref{l:O:Exp:dual} are exactly \cite[Prop.~2.6.1.2 (ii)-(iii)]{EdmundsTriebel-book}. Property \ref{l:O:Hol:LlogL} is given by \cite[Thm.~2.3]{ONeil65}. The continuous embeddings in \ref{l:O:H1:Exp:ct} are given by \cite[Thm.~8.16]{AdamsFournier03}. 
\smallskip

\noindent
Property \ref{l:O:Hol:L1} follows from \cite[Thm.~8.11]{AdamsFournier03}, since  $\phi_1 (t) := e^{t} - 1$ and $\phi^1 (t) := t (\log(e + t))$ are complementary functions. To be precise, \cite[Thm.~8.11]{AdamsFournier03} is valid for $\psi_1(t) := e^{t} - t - 1$ and $\psi^1 (t) := (t+1)(\log(1 + t))-t$, which are complementary $N$- functions (unlike $\phi_1$ and $\phi^1$). In a bounded domain, however, since $\phi_1, \psi_1$ and $\phi^1, \psi^1$ have the same behaviour at infinity, the corresponding Orlicz spaces coincide, and the corresponding norms are equivalent. 
\smallskip

\noindent
Finally we establish the compact embedding in \ref{l:O:H1:Exp:cp}. First of all we 
recall that the natural embedding
$$
\textrm{id}_{\T^d}: H^{d/2} (\R^d) \to \Exp_\alpha (\T^d), \quad u\mapsto u_{|\T^d}
$$
is compact for every $1<\alpha< 2$ (this is true for every $\Omega\subset \R^d$ bounded, see e.g. \cite{Kuhn03}). We also recall that there exists an extension operator $T: H^{d/2}(\T^d) \to H^{d/2}(\R^d)$ with $\| Tu \|_{H^{d/2}(\R^d)} \leq C_d \|u\|_{H^{d/2}(\T^d)}$ for all $u \in H^{d/2}(\T^d)$ (see, e.g., \cite[Thm.~2.2]{Rych99}). Combining these two results we obtain the sought compact embedding of  $H^{d/2} (\T^d)$ into $\Exp_\alpha (\T^d)$, since we can write every $u\in H^{d/2} (\T^d)$ as $u= \textrm{id}_{\T^d}(Tu)$.
\end{proof}

\begin{lem}[Aubin-Lions-Simon] \label{l:ALS}
Let $1 \leq p, q \leq \infty$ and $X_0, X , X_1$ be Banach spaces such that $X_0 \subset \subset X \hookrightarrow X_1$. Then 
$$
\big\{ u \in L^p(0,T;X_0) : \partial_t u \in L^q(0,T;X_1) \big\}
\text{ is relatively compact in }\left\{ \begin{aligned}
  &L^p (0,T;X) 
  &&p < \infty, \\
  &C ([0,T];X) 
  &&p = \infty. 
\end{aligned} \right.
$$ 
\end{lem}
\bigskip

%%%%%%%%%%%%%%%%%%%%%%%%%%%%%%%%%%%%%%%%%%%%%%%%%%%

\noindent\textbf{Acknowledgments}.
PvM is supported by the International Research Fellowship of the Japanese Society for the Promotion of Science, together with the JSPS KAKENHI grant 15F15019. MAP and LS acknowledge support from NWO grant 639.033.008. LS acknowledges support by the EPSRC under the grant EP/N035631/1.

The authors declare to have no conflict of interest.

\bibliographystyle{alpha} 
%\bibliography{refsArticleUpscaling}
\newcommand{\etalchar}[1]{$^{#1}$}

\end{document}